\documentclass[11pt]{amsart}

\usepackage{epigamath-voisin}

\usepackage[english]{babel}

\numberwithin{equation}{section}

\usepackage[shortlabels,inline]{enumitem}

\usepackage[all]{xy}
\usepackage{amsmath}
\usepackage{amssymb}
\usepackage{mathrsfs}   
\usepackage{titlesec} 
\usepackage{xcolor}
\usepackage{upgreek}
\usepackage{calc}
\usepackage[normalem]{ulem}
\usepackage{caption}
\usepackage{amsthm}
\usepackage{amsfonts}

\usepackage{macro-longdash}

\theoremstyle{plain}
\newtheorem{theorem}{Theorem}[section]
\newtheorem{lemma}[theorem]{Lemma}
\newtheorem{proposition}[theorem]{Proposition}
\newtheorem{corollary}[theorem]{Corollary}

\theoremstyle{definition}
\newtheorem{definition}[theorem]{Definition}
\newtheorem{noname}[theorem]{}

\theoremstyle{remark}
\newtheorem{remark}[theorem]{Remark}

\newtheorem{example}[theorem]{Example}

\newtheorem{notation}[theorem]{Notation}

\newcommand{\C}{\mathbf{C}}

\newcommand{\Q}{\mathbf{Q}}

\newcommand{\N}{\mathbf{N}}
\renewcommand{\P}{\mathbf{P}}

\newcommand{\PH}{\mathbf{P}\kern -.05em \mathrm{H}}

\newcommand{\A}{\mathbf{A}}
\newcommand{\F}{\mathbf{F}}
\renewcommand{\D}{\mathbf{D}}
\newcommand{\T}{\mathbf{T}}

\renewcommand{\O}{\mathcal{O}}
\newcommand{\I}{\mathcal{I}}
\renewcommand{\H}{\mathcal{H}}

\renewcommand{\L}{\mathcal{L}}
\newcommand{\M}{\mathcal{M}}

\newcommand{\Cliff}{\mathrm{Cliff}}

\newcommand{\Ext}{\mathrm{Ext}}

\DeclareMathOperator{\Sym}{\mathrm{Sym}}

\newcommand{\red}{\mathrm{red}} 

\newcommand{\cHom}{{\mathcal H}\kern -.08em om} 
\newcommand{\cExt}{{\mathcal E}\kern -.1em xt} 

\newcommand{\vect}[1]{\langle #1 \rangle} 
\newcommand{\lineq}{\sim} 

\DeclareMathOperator{\cork}{cork}

\newcommand{\Iff}{\Longleftrightarrow}

\renewcommand{\L}{\mathcal{L}}

\renewcommand{\l}{\ell}

\newcommand{\dlbrack}{[ \kern -.4ex [}
\newcommand{\drbrack}{] \kern -.4ex ]}

\newcommand{\trsp}[1]{\vphantom{#1}^{\mathsf T\!} #1}

\makeatletter
\newcommand{\restr}[2]{\left. #1 \right| _{#2}}
\def\@restrpar[#1]#2{\left. (#1) \right| _{#2}}
\def\@restrst#1#2{\left. #1 \right| _{#2}}
\def\restr{\@ifnextchar[{\@restrpar}{\@restrst}}

\def\@orthpar[#1]{(#1)^\perp}
\def\@orthst#1{#1^\perp}
\def\orth{\@ifnextchar[{\@orthpar}{\@orthst}}

\def\@dualpar[#1]{(#1)^\vee}
\def\@dualst#1{#1^\vee}
\def\dual{\@ifnextchar[{\@dualpar}{\@dualst}}

\makeatother

\newcommand{\cC}{\mathcal{C}}
\newcommand{\cS}{\mathcal{S}}
\newcommand{\cSC}{\mathcal{SC}}

\newcommand{\cH}{\mathcal{H}}
\newcommand{\cL}{\mathcal{L}}
\newcommand{\cJ}{\mathcal{J}}

\newcommand{\bx}{\mathbf x}
\newcommand{\by}{\mathbf y}
\newcommand{\bef}{\mathbf f}

\newcommand{\bh}{\mathbf h}
\newcommand{\bt}{\mathbf t}

\newcommand{\fe}{\mathfrak{e}}
\newcommand{\ff}{\mathfrak{f}}

\newcommand{\fg}{\mathfrak{g}}
\newcommand{\fh}{\mathfrak{h}}

\renewcommand{\epsilon}{\varepsilon}
\renewcommand{\geq}{\geqslant}
\renewcommand{\leq}{\leqslant}

\def\subset{\subseteq}
\renewcommand{\emptyset}{\varnothing}

\newcommand{\et}{\quad \text{and} \quad}
\newcommand{\ie}{\textit{i.e.} } 
\newcommand{\cf}{\textit{cf.} } 
\newcommand{\eg}{\textit{e.g.} }
\newcommand{\noeud}{n{\oe}ud}
\newcommand{\noeuds}{n{\oe}uds}

\makeatletter
\def\noeud{\@ifnextchar.{n{\oe}ud}{\@ifnextchar,{n{\oe}ud}{n{\oe}ud\ }}}
\def\noeuds{\@ifnextchar.{n{\oe}uds}{\@ifnextchar,{n{\oe}uds}{n{\oe}uds\ }}}
\makeatother

\def\?{?\kern -.08em ?}
\def\wtf{?\kern -.08em !}

\newcommand{\p}{\mathrm{pr}}

\setlist[enumerate,1]{label={\rm(\roman*)}, ref={\rm\roman*}}

\newlist{a-enumerate}{enumerate}{2}
\setlist[a-enumerate,1]{label={\rm(\alph*)}, ref={\rm\alph*}}

\newcommand{\lra}{\longrightarrow}

\newcommand{\supth}[1]{\ensuremath{#1^{\mathrm{th}}}}

\usepackage{url}

\YearArticle{2024} \EpigaArticleNr{16} \ReceivedOn{April 17, 2023}
\InFinalFormOn{January 8, 2024}
\AcceptedOn{January 27, 2024}

\title{Extensions of curves with high degree \\ with respect to the genus}
\titlemark{Extensions of curves with high degree with respect to the genus}

\author{Ciro Ciliberto}
\address{Dipartimento di Matematica,
Universit{\`a} degli Studi di Roma Tor Vergata, 
Via della Ricerca Scientifica, 
00133 Roma, Italy}
\email{cilibert@mat.uniroma2.it}
\author{Thomas Dedieu}
\address{Institut de Math{\'e}matiques de Toulouse,  UMR5219, 
Universit{\'e} de Toulouse,  CNRS, 
UPS IMT, F-31062 Toulouse Cedex 9, France}
\email{thomas.dedieu@math.univ-toulouse.fr}

\authormark{C.~Ciliberto and T.~Dedieu}

\AbstractInEnglish{We classify linearly normal surfaces $S \subset
  \P^{r+1}$ of degree $d$ such that $4g-4 \leq d \leq 4g+4$, where
  $g>1$ is the sectional genus (it is a classical result that for
  larger $d$ there are only cones). We apply this to the study of the
  extension theory of pluricanonical curves and genus $3$ curves
  whenever they have property $N_2$, using and slightly expanding
  the theory of integration of ribbons of the authors and Sernesi.
  We compute the corank of the relevant Gaussian maps, and we show
  that all ribbons over such curves are integrable, and thus there
  exists a universal extension.

We carry out a similar programme for linearly normal hyperelliptic
curves of degree $d\geq 2g+3$. We classify surfaces having such a
curve $C$ as a hyperplane section, compute the corank of the relevant
Gaussian maps, and prove that all ribbons over $C$ are integrable if
and only if $d=2g+3$. In the latter case we obtain the existence of a
universal extension.}



\acknowledgement{CC is a member of GNSAGA of INdAM. He
acknowledges the MIUR Excellence Department Project awarded to the
Department of Mathematics, University of Rome Tor Vergata, CUP
E83C18000100006.
ThD acknowledges support from the ANR project FanoHK, grant
ANR-20-CE40-0023.}

\dedication{Dedicated to Claire Voisin on the occasion of her birthday,
with admiration and heartful gratitude} 

\begin{document}


\maketitle

\begin{prelims}

\DisplayAbstractInEnglish

\bigskip

\DisplayKeyWords

\medskip

\DisplayMSCclass

\end{prelims}


\newpage

\setcounter{tocdepth}{1}

\tableofcontents


\section{Introduction}

In this article we study the extensions of certain non-special curves
of genus $g\geq 2$.
We shall consider smooth, irreducible, linearly normal projective
curves $C$ of genus $g\geq 2$ and degree $d$ in $\P^r$; it will
always be the case that $d>2g-2$, hence
$r=d-g$. The \emph{extensions} we want to study are surfaces $S
\subset \P^{r+1}$
having $C$ as a hyperplane section, and more
generally $(c+1)$-dimensional varieties $X \subset \P^{r+c}$ having the curve
$C$ as a section by a linear space.
An extension of $C$ is \emph{non-trivial} if it is not a cone.
In this article we will be interested in non-trivial extensions.

A classical theorem by C.~Segre says that if a surface extension of $C$
is a scroll, then it is actually a cone over~$C$.  We shall give this
theorem a modern proof in the present text; see \cite[Section
  2]{CSegre-torino} for the original proof.  On the other hand, a
theorem by Hartshorne \cite[Theorem (4.1)]{hartshorne-bound} says that
if $d > 4g+4$, then $S$ is a scroll; together with the previous result
by C.~Segre, this implies that it is a cone.  The upshot is that for our
study, we can assume without loss of generality that $d \leq 4g+4$.
Our first result is the classification of surface extensions of a
curve as above in the range $4g-4 \leq d \leq 4g+4$.

The following notation will be used throughout the text.

\begin{notation}
\label{nz:scrolls}
For all $e\in \N$, we let $\F_e$ be the rational ruled surface
$\P(\O_{\P^1}\oplus \O_{\P^1}(-e))$, and we denote by
$E$ the section with self-intersection $-e$
(in the case $e=0$, this has to be taken with a grain of salt) and
by $F$ the class of the fibres, and we set $H=E+eF$.
\end{notation}

\begin{theorem}
\label{t:hideg}
Let $S\subset \P^{r+1}$ be a non-degenerate irreducible, projective surface
of degree $d \geq 4g-4$ whose general hyperplane section $C$ is smooth, of
genus $g\geq 2$, and linearly normal.
If\, $S$ is not a cone, then one of the following holds:
\begin{a-enumerate}
\item
\label{ghd:biell}
  $S$ is the image by the Veronese map $v_2$ of a cone over an
  elliptic normal curve of degree $g-1$,
  and the hyperplane sections of\, $S$ are bielliptic bicanonical
  curves, as in Example~\ref{ex:biell}. 
\item
\label{ghd:plane}
  $S$ is a rational surface represented by a linear system of
smooth
  plane $\delta$-ics, $4 \leq \delta \leq 6$, as in
  Example~\ref{ex:plane}. 
\item
\label{ghd:DP}
   $S$ is the image by the Veronese map $v_2$ of a Del Pezzo
  surface, as in Example~\ref{ex:delp}. 
\item
\label{ghd:hyperell}
  $S$ is a rational surface with hyperelliptic sections,
  represented by a linear subsystem of\, 
  $|2H+(g+1-e)F|$ on $\F_e$, as
  in Example~\ref{ex:hyperell};
\item
\label{ghd:trigonal}
   $S$ is a rational surface with trigonal sections,
  represented by a linear subsystem of 
  \linebreak
  $|3H+\frac 1 2 (g-3e+2)F|$ on $\F_e$,
  and $g\leq 10$, as in
  Example~\ref{ex:trig}. 
\end{a-enumerate}
\end{theorem}

In the above statement, linear (sub)systems
(in cases \eqref{ghd:plane},
\eqref{ghd:hyperell},
and \eqref{ghd:trigonal})
are defined by simple base points, possibly infinitely near but along
a curvilinear scheme.
Our proof of Theorem~\ref{t:hideg} is based on a
careful analysis of the adjoint system of $C$ on $S$.

\bigskip
On the other hand, if $C \subset \P^r$ has property $N_2$ (see
Theorem~\ref{t:green}), then its extension theory is governed by
ribbons.  The latter are non-reduced schemes supported on $C$ that are
potential first infinitesimal neighbourhoods of $C$ in an extension $S
\subset \P^{r+1}$.  The salient points of the theory, which we recall
from \cite{cds} and slightly expand in Section~\ref{S:ribb'n'ext}, are
the following: \begin{enumerate*}
\item\label{item1} isomorphism classes of non-trivial ribbons are
  parametrized by the projective space $\P(\ker (\trsp
  \gamma_{C,L}))$, where $L=\restr {\O_{\P^r}(1)} C$ and
  $\gamma_{C,L}$ is a Gaussian map, the definition of which is
  recalled in Section~\ref{def:gaussian};
\item\label{item2} each ribbon may be the first infinitesimal
  neighbourhood of $C$ in at most one extension; in particular, $C$ may
  have a non-trivial extension only if $\gamma_{C,L}$ is not
  surjective;
\item\label{item3} if all ribbons in $\P(\ker (\trsp \gamma_{C,L}))$
  can be realized as first infinitesimal neighbourhoods of $C$ in an
  extension, then there exists a \emph{universal extension} of $C$,
  \ie
  a $(c+1)$-dimensional variety $X \subset \P^{r+c}$, $c = \cork
  (\gamma_{C,L})$, having $C$ as a curve section and such that all
  surface extensions of $C$ are realized in a unique way as a section
  of $X$ by some $(r+1)$-dimensional linear space containing $C$.
\end{enumerate*}

We shall apply our classification theorem, Theorem~\ref{t:hideg}, to
study the existence of non-trivial extensions of polarized curves and,
in favourable cases, 
prove the existence of universal extensions, as follows.
First, we have found some situations in which we can compute the
dimension of $\P(\ker (\trsp \gamma_{C,L}))$, which in general is a
very difficult task.
Then our idea is to consider the family of all possible extensions of
the curve $C$, using our classification theorem, which gives the
dimension of the locus in $\P(\ker (\trsp \gamma_{C,L}))$ of those
ribbons corresponding to an actual extension of $(C,L)$ by the unicity
property~\eqref{item2} above. When the two dimensions match, we can conclude
that there exists a universal extension by the general
Theorem~\ref{t:gnl-ribb}. \looseness=1

The dimension of the space $\P(\ker (\trsp \gamma_{C,L}))$ is $\cork
(\gamma_{C,L})-1$, where $\cork (\gamma_{C,L})$ denotes the corank of
the map $\gamma_{C,L}$, \ie the codimension of its image in
$H^0(C,2K_C+L)$; see Section~\ref{def:gaussian}.  We compute it in the
following situations.

\begin{theorem}
\label{t:cork-iperell+g3}
Let $C$ be a smooth projective curve of genus $g\geq 2$ and
$L$ a line bundle on $C$ of degree $d$.
\begin{a-enumerate}
\item Assume $C$ is hyperelliptic; if either $d \geq 2g+3$,
  or $d\geq g+4$ and $L$ is general, then
  $\cork (\gamma_{C,L}) = 2g+2$.
\item Assume $g=3$ and $C$ is non-hyperelliptic;
  if either $d \geq 2g=6$,
  or $d\geq g+1=4$ and $L$ is general,
  then
  $\cork(\gamma_{C,L}) = h^0(C,4K_C-L)$.
\end{a-enumerate}
\end{theorem}

In fact the above result is an application of
Proposition~\ref{pr:cork-general}, which enables one to compute the
corank of $\gamma_{C,L}$ in virtually any situation, provided $C$
either is hyperelliptic or has genus $3$.
For the genus $3$ case, we essentially give another proof to
an earlier result by Knutsen and Lopez
\cite[Proposition~2.9(a)]{kl}.
The following, on the other hand, is essentially a compilation of
previously known results.

\begin{theorem}
\label{t:cork-plurican}
Let $C$ be a smooth projective curve of genus $g\geq 5$,
non-hyperelliptic, and
$L = mK_C$ for some integer $m>1$.
Then $\cork (\gamma_{C,mK_C}) = 0$ if $m>2$ or $\Cliff(C)>2$.

\noindent
If\, $\Cliff(C)=1$, then either
\begin{a-enumerate}
\item $C$ is trigonal, and then
  $\cork (\gamma_{C,2K_C}) = h^0(K_C-(g-4)\fg)$, with $\fg$ the class of the
  $g^1_3$; or
\item $C$ is a plane quintic, and then
   $\cork (\gamma_{C,2K_C}) = h^0(\P^2,-2K_{\P^2}-C) = 3$.
\end{a-enumerate}

\noindent
If\, $\Cliff(C)=2$, then $\cork (\gamma_{C,mK_C}) = 0$ except in the
following cases: 
\begin{a-enumerate}
\item\label{cliff2-a} $g=5$, and then $\cork (\gamma_{C,2K_C}) = 3$;
\item\label{cliff2-b} $C$ is a bi-anticanonical divisor in a Del Pezzo surface $X$, and
  then we have 
  $\cork (\gamma_{C,2K_C}) =
  \linebreak
  h^0(X,-2K_X-C) = 1$;
\item\label{cliff2-c} $C$ is bielliptic, and then
  $\cork (\gamma_{C,2K_C}) = 1$.
\end{a-enumerate}
\end{theorem}

\noindent
The case of curves of genus $g\leq 4$ is elementary; see
\eqref{eq:cork-g3,4}.

\bigskip
The next stage of our programme
is to examine in all three cases above
(genus $3$ curves, hyperelliptic curves, and pluricanonical curves)
the families of all surface extensions of a given polarized curve
$(C,L)$.
For pluricanonical curves, Theorem~\ref{t:hideg} above tells us all
the surfaces we need to consider.
In general we will assume that $\deg(L)\geq 2g+3$,
in order for property $N_2$ for $(C,L)$
(which is needed to apply the theory of ribbons and extensions)
to hold by Green's theorem, Theorem~\ref{t:green}.
For pluricanonical curves, this condition is automatic except for a
few sporadic cases in genus $g\leq 3$.
For an arbitrary polarized curve of genus $3$, this condition is
stronger than $\deg(L) \geq 4g-4$, so that again Theorem~\ref{t:hideg}
tells us all the surfaces we need to consider.

For a polarized hyperelliptic curve $(C,L)$ of arbitrary genus, 
however, we need a stronger classification result. We prove the
following, which extends classical results by
Castelnuovo, \cf \cite{castelnuovo-iperell}, and more
recent ones by Serrano, \cf \cite{serrano}, 
and Sommese--Van de Ven, \cf \cite{sommese-vdVen}
(see Section~\ref{p:castelnuovo-storia} for more comments on these results).

\begin{theorem}
\label{t:hyperell-notrat}
Let $C \subset \P^{d-g}$ be a linearly normal hyperelliptic curve of
genus $g\geq 2$ and degree $d \geq 10$, unless $g=2$ or $3$, in which
case we only make the looser assumption that $d\geq 2g+3$.
For all surfaces $S \subset \P^{d-g+1}$ having $C$ as a hyperplane
section, if\, $S$ is not a cone, then it is rational and ruled by conics.
In particular, its general hyperplane section is hyperelliptic.
\end{theorem}

\begin{corollary}
\label{c:he_class}
In the setting of Theorem~\ref{t:hyperell-notrat}, the surface $S$ is
represented by a linear subsystem of\, $|2H+(g+1-e)F|$ on $\F_e$, as
in case \eqref{ghd:hyperell} of Theorem~\ref{t:hideg}.
\end{corollary}

\noindent
Our main tool in proving Theorem~\ref{t:hyperell-notrat} is
the Reider and Beltrametti--Sommese theorem, Theorem~\ref{t:R/B-S}.

\bigskip
Finally, we can complete our programme, to the effect that we obtain
the following results.

\begin{theorem}\label{t:ext_g3}
Let $(C,L)$ be a non-hyperelliptic polarized curve of genus $g=3$ and
degree $d\geq 2g+3 =9$. Then the following hold: 
\begin{enumerate}
\item 
\label{t:ext_g3:cond}
There exists a non-trivial extension of the polarized curve $(C,L)$
if and only if
there exist points $p_1,\ldots,p_{16-d} \in C$ such that
\smash{$L = 4K_C - \sum_{i=1}^{16-d} p_i$} $($in particular, $d\leq 16)$.
\item 
\label{t:ext_g3:all}
Every ribbon in $\P( \ker (\trsp \gamma_{C,L}))$ is the first
infinitesimal neighbourhood of $C$ in some extension of\, $(C,L)$; hence
there exists a universal extension of\, $(C,L)$.
\end{enumerate}
\end{theorem}

We refer to Section~\ref{p:univ-ext-g3} for a discussion of the universal
extensions of polarized genus $3$ curves.
We also include an analysis of what happens for degrees below
$2g+3$, in which case property $N_2$ is no longer implied by Green's
theorem, and many results about ribbons and extensions are no longer
available; in particular, a given ribbon may \textit{a priori} be the first
infinitesimal neighbourhood of $C$ in several different extensions.
Notably we give examples of polarized curves having two
distinct families of extensions, one of the expected dimension
$\cork(\gamma_{C,L})-1$ and one superabundant,
which would be impossible if property $N_2$ held.

\begin{theorem}
\label{t:ext_he}
Let $(C,L)$ be a polarized hyperelliptic curve
of genus $g$ and degree $d\geq 2g+3$.
\begin{a-enumerate}
\item 
\label{t:ext_he:ext}
If\, $d \leq 4g+4$, then there exists a non-trivial extension of $(C,L)$.
\item 
\label{t:ext_he:univ}
If\, $d=2g+3$, then every ribbon in $\P( \ker (\trsp \gamma_{C,L}))$ is
the first infinitesimal neighbourhood of\, $C$ in some extension of\,
$(C,L)$, hence there exists a universal extension of\, $(C,L)$,
of degree $2g+3$ and dimension $2g+3$ in $\P^{3g+5}$.
\item 
\label{t:ext_he:obstr}
If\, $d>2g+3$, then for general $(C,L)$ there exist ribbons in $\P( \ker
(\trsp \gamma_{C,L}))$ which may not be realized as
first infinitesimal neighbourhoods of\, $C$ in some extension of\,
$(C,L)$.
\end{a-enumerate}
\end{theorem}

When $d=4g+4$, $(C,L)$ in general has only finitely many extensions
but more than one; thus there cannot exist a universal extension of
$(C,L)$. Also, in this case we analyze briefly the situation for
degrees below $2g+3$, and we find that in this case the extensions
form a superabundant family; \ie this family has dimension greater
than $\cork(\gamma_{C,L})-1$.
This would be impossible if property $N_2$ held.

\begin{theorem}
\label{t:plurican}
Let $(C,mK_C)$ be a pluricanonical curve, and assume that either
$g\geq 4$ and $m\geq 2$, or $g=3$ and $m \geq 3$.
Then there exists a non-trivial extension of the polarized curve
$(C,mK_C)$ if and only if\, $(C,mK_C)$ falls into one of the cases of\,
Theorem~\ref{t:cork-plurican} or Section~\ref{p:ci-cork}
in which $\cork (\gamma_{C,mK_C}) \neq 0$
$($in particular, either $g\leq 10$, or $C$ is bielliptic$)$.

Every ribbon in $\P( \ker (\trsp \gamma_{C,mK_C}))$ is
the first infinitesimal neighbourhood of\, $C$ in some extension of\,
$(C,mK_C)$; hence there exists a universal extension of\, $(C,mK_C)$.
\end{theorem}

We refer to Section~\ref{s:univ-plurican} for a discussion of the
universal extensions of pluricanonical curves. Except in the trigonal
case, we provide an 
explicit construction.

\bigskip
The organization of the text is as follows.
In Section~\ref{S:CSegre} we recall general results on projective
curves, revisit the C.~Segre theorem, and recall the Hartshorne bound.
Section~\ref{S:class-hideg} is devoted to the proof of the
classification theorem, Theorem~\ref{t:hideg},
and Section~\ref{S:class-hyperell} to that of the classification
theorem for hyperelliptic curves, Theorem~\ref{t:hyperell-notrat}.
In Section~\ref{S:coker} we recall the definition of the Gaussian map
$\gamma_{C,L}$ and compute its corank in a number of cases, thus
proving Theorem~\ref{t:cork-iperell+g3}.
In Section~\ref{S:ribb'n'ext} we recall the theory of ribbons and
extensions, and provide all necessary material for its application
as described in the introduction. These applications are to polarized
genus $3$ curves (in Section~\ref{S:ext-g3}),
to polarized hyperelliptic curves (in Section~\ref{S:ext-hyperell}),
and to pluricanonical curves (in Section~\ref{S:bican}).

We work over the field $\C$ of complex numbers throughout. We use the
symbols '$\equiv$' and '$\lineq$' to denote numerical and linear
equivalence, respectively.

\subsection*{Acknowledgments}
We thank Andreas Knutsen and Angelo Lopez for useful comments.

\section{General preliminary results}
\label{S:CSegre}

\subsection{Results on projective curves}

We will need the following results.  The first one is an improvement
on a theorem by Castelnuovo; see \cite{castelnuovo-suimultipli} and
\cite[Theorem~(1.11)]{ciliberto83}.

\begin{theorem}[Castelnuovo]
\label{t:castelnuovo}
Let $C$ be a smooth curve and $L\to C$ be a globally generated line
bundle.  If the image of the map associated to the complete linear
series $|L|$ is not a rational curve, or if\, $|L|$ is a pencil, then
the multiplication map
\[
  H^0(L) \otimes H^0(K_C)
  \lra H^0(K_C+L)
\]
is surjective.
\end{theorem}

\noindent
The other one is due to Green.

\begin{theorem}[\cf \protect{\cite[Theorem~4.a.1]{green84}}]\label{t:green}
Let $C$ be a smooth curve of any genus $g$ and $L\to C$ a line bundle of
degree $d$. For all $k\geq 0$, if
\[
d \geq 2g+1+k,
\]
then $L$ has property $N_k$, \ie
\begin{enumerate}
\item $L$ defines a projectively normal embedding of\, $C$, and
\item if\, $k\geq 1$, the ideal of\, $C$ in this embedding is generated by
quadrics and all syzygies are generated by linear syzygies up to the
$\supth{k}$ step.
\end{enumerate}
\end{theorem}

\subsection{Scrolls as extensions of linearly normal curves}

The main object of this subsection is to discuss
Theorem~\ref{l:CSegre} below.
We distinguish between \emph{ruled surfaces}, by which me mean
``abstract ruled surfaces'', \ie surfaces $S$ equipped with a locally
trivial morphism $S\to C$ onto a smooth curve whose fibres are $\P^1$,
and \emph{scrolls}, by which we mean a ruled surface embedded in some
projective space in such a way that the fibres are lines.

\begin{theorem}
\label{l:CSegre}
Let $C \subset \P^{n-1}$ be a smooth linearly normal and
non-degenerate curve of genus $g>0$.
Assume there exists a scroll $\Sigma \subset \P^n$ such that $C$ is a
hyperplane section of\, $\Sigma$.
Then the scroll $\Sigma$ is necessarily a cone.
\end{theorem}

This first appeared in \cite[Section~2]{CSegre-torino}
by C.~Segre, and
also later in \cite[Section 14]{CSegre-mathann} by the same author,
under the additional
assumption that $\O_C(1)$ is non-special.
In the particular case when $C$ is a canonical curve, this is
\cite[Theorem~III.2.1]{epema-thesis}.

We shall need the following lemma for the proof.

\begin{lemma}[\cf \protect{\cite[Lemme~1]{beauville-merindol}}]
\label{l:B-M}
Let $C$ be a smooth curve, and let
\begin{equation}
\label{ext:BM}
0 \lra E' \lra E \lra E'' \lra 0
\end{equation}
be an exact sequence of vector bundles on $C$.
Assume that
\begin{enumerate}
\item\label{l:B-M-1} the boundary map
  $\partial \colon H^0(E'') \to H^1(E')$
  is zero, and
\item\label{l:B-M-2} the multiplication map 
$\alpha\colon H^0(E'') \otimes H^0((E')^\vee \otimes \omega_C)
\to  H^0(E'' \otimes (E')^\vee \otimes \omega_C)$
is surjective.
\end{enumerate}
Then the exact sequence \eqref{ext:BM} is split.
\end{lemma}

\begin{proof}[Proof of Theorem~\ref{l:CSegre}]
Let $f\colon S \to \Sigma$ be the minimal resolution of singularities. The
surface $S$ is ruled with $C$ as a section of the ruling,
and $\Sigma$ is a scroll.
Hence there exists a rank 
$2$ vector bundle $\mathcal{E}$ on $C$ such that $S\cong
\P(\mathcal{E})$ and the map $S\to \Sigma \subset \P^n$ is given by a
linear subsystem of $|H^0(\O_{\P({\mathcal E})}(1))|$.
Moreover, there is the following exact sequence of locally free
sheaves on $C$:
\begin{equation}
\label{ext-segre}
\xymatrix@C=0pt@R=6pt{
0 \lra & 
\!\!\!\O_C(1) \otimes N_{C/S}^{-1} \!\!\!
\ar@{=}[d]
&\lra 
{\mathcal E} \lra 
\O_C(1) \lra 0; \\
& \O_C
}
\end{equation}
see \cite[Proposition~V.2.6]{hartshorne}.
We shall use Lemma~\ref{l:B-M} to show that this exact sequence
is split. Then $\mathcal{E} = \O_C\oplus \O_C(1)$,  the linear
system $|H^0(\O_{\P({\mathcal E})}(1))|$ contracts the section
corresponding to the trivial quotient $\mathcal{E} \to \O_C$, and
$\Sigma$ is a cone, which is the result we wanted to prove.

It thus only remains to apply Lemma~\ref{l:B-M} to the exact sequence
\eqref{ext-segre}. Condition~\eqref{l:B-M-2} is satisfied by Castelnuovo's theorem, 
Theorem~\ref{t:castelnuovo}; \ie the multiplication map
$H^0(L) \otimes H^0(K_C) \to H^0(K_C+L)$, where $L=\O_C(1)$, is surjective.
To see that condition~\eqref{l:B-M-1}  of the lemma holds as well, we write the
long exact sequence associated to \eqref{ext-segre}:
\begin{equation*}
\xymatrix@C=0pt@R=6pt{
0 \lra H^0(\O_C) \lra 
& \!\!\!H^0({\mathcal E})\!\!\!
\ar@{=}[d] 
& \lra H^0(\O_C(1))
\lra &
\!\!\!H^1(\O_C)\!\!\!
\ar@{=}[d] 
& \lra
& \!\!\!H^1({\mathcal E})\!\!\!
\ar@{=}[d] 
& \lra
H^1(\O_C(1)) \lra 0;
\\
& H^0(\O_{S}(1))
&& H^1(\O_{S})
&& H^1(\O_{S}(1))
}
\end{equation*}
for the vertical identifications, see \eg
\cite[Lemma~V.2.4]{hartshorne};
$\O_S(1)$ stands for $\O_{\P({\mathcal E})}(1)$.
It follows from the fact that $C \subset \P^{n-1}$ is linearly normal
that 
$H^0(\O_{\Sigma}(1)) \to H^0(\O_C(1))$ is surjective,
hence
$H^0(\O_{S}(1)) \to H^0(\O_C(1))$ is surjective, and thus the boundary
map
$H^0(\O_C(1))\to H^1(\O_C)$ is zero; \ie condition~\eqref{l:B-M-1} holds.
We may thus apply Lemma~\ref{l:B-M}, and, as explained above, this
concludes the proof of the theorem.
\end{proof}

\subsection{The Hartshorne bound}

\begin{theorem}[\cf \cite{hartshorne-bound}]
\label{t:hartshorne}
Let $C$ be a smooth curve of genus $g$ sitting in a smooth surface $S$.
If\, $C^2 >4g+5$, then there exist a ruled surface $\Sigma$ having $C$
as a section and a birational map $S \dashrightarrow \Sigma$ which is
an isomorphism on an open subset containing $C$.

If\, $C^2=4g+5$, the only other possibility is that
there is a birational map $S \dashrightarrow \P^2$ which is an
isomorphism on an open subset containing $C$ and identifies $C$ with
a cubic curve.
\end{theorem}

This result is also a particular case of \cite[Theorem~A]{horowitz}.

\begin{corollary}
\label{c:hartshorne}
Let $C \subset \P^n$ be a smooth curve of genus $g>1$ and degree $d$,
non-degenerate and linearly normal.  If $d > 4g+4$, then every
extension of\, $C$ is trivial.

If $g=1$ and $d=4g+5$, the only possibility for $C$ to have a
non-trivial extension
is that it is a hyperplane section of the Veronese surface $v_3(\P^2)
\subset \P^9$.
\end{corollary}

\begin{proof}
Let $S \subset \P^{n+1}$ be a surface having $C$ as a hyperplane
section.  Then $\deg(S)=\deg(C)=d$; hence $C^2=d$ as a divisor in $S$.
We consider a minimal desingularization $\pi\colon S'\to S$, and,
abusing notation, we still denote by $C$ its proper transform on $S'$.
Note that $S$ may have at most isolated singularities, and by the
minimality of the resolution, there is no irreducible $(-1)$-curve
$\Gamma$ on $S'$ such that $C\cdot \Gamma=0$.

By Hartshorne's theorem, Theorem~\ref{t:hartshorne} above, there are
two cases to be considered. First assume that there exists a
birational map $S' \dashrightarrow \P(E)$ which is an isomorphism on
an open subset $U \subset S'$ containing $C$, where $E$ is a rank $2$
vector bundle over $C$.  The pull-back to $S'$ of the linear system
$|C|$ on $S$ is the complete linear system $|C|$ on $S'$ because $C
\subset \P^n$ is linearly normal, and it is base-point-free.  In turn,
the image on $\P(E)$ of this system is again the complete linear
system $|C|$ because $S' \dashrightarrow \P(E)$ is an isomorphism on
$U$.  The upshot is that $S$ is the image of $\P(E)$ defined by this
linear system, and by Theorem~\ref{l:CSegre} this is a cone.

In the remaining case of Hartshorne's theorem, which may occur only if
$g=1$ and $d=4g+5=9$, similar arguments show that $S$ is the image of
$\P^2$ by the complete linear system of plane cubics.
\end{proof}

\section{Classification of extensions with high degree}
\label{S:class-hideg}

This section is devoted to Theorem~\ref{t:hideg}.  We first expand on
the description of the items in the classification and then give the
proof.

\subsection{Detailed description of the items in the classification}

Let $S \subset \P^N$ be a degree $d$ surface, of sectional (geometric)
genus $g$.  We call \emph{simple internal projection} of~$S$ a surface
$S'\subset \P^{N'}$ obtained by projecting $S$ from a curvilinear
subscheme $Z$ of length $b$ supported on the smooth locus of $S$,
where $N'=N-\dim (\vect Z)$, such that the projection map is
birational.  We recall that a scheme $Z$ is \emph{curvilinear} if for
all points $p$ in the support of $Z$, the Zariski tangent space of $Z$
at $p$ has dimension at most $1$.

For a simple internal projection $S'$ of $S$ as above, one has
$\deg(S') = d-b$, and $S'$ has the same sectional genus $g$ as $S$.
Note that if $d-b \geq 2g+1$ and $S$ is regular and linearly normal,
then any projection from a curvilinear subscheme $Z$ of length $b$
supported on the smooth locus of $S$ is a simple internal projection,
and $N'=N-b$, for in this case the linear system of hyperplane
sections of $S$ passing through $Z$ restricts on its general member to
a complete, non-special, very ample linear system.

\begin{example}\label{ex:biell} Let $C$ be a bielliptic curve of genus
$g\geq 4$. Then the canonical model of $C$ in $\P^{g-1}$ sits on a
cone $X$ with vertex a point $p$ over a normal elliptic curve $E$ of
degree $g-1$ in a hyperplane $\Pi$ of $\P^{g-1}$ not containing $p$,
and $C$ is the complete intersection of $X$ with a quadric in
$\P^{g-1}$. The bielliptic involution is the restriction to $C$ of the
projection from $p$ to $\Pi$.

Note that the minimal resolution of $X$ is the projective bundle
$\P(\O_E\oplus \L)$, where $\L$ is the hyperplane bundle of $E$ in
$\Pi\cong \P^{g-2}$. The map $\P(\O_E\oplus \L)\to X$ is induced by
the $\O(1)$ bundle on $\P(\O_E\oplus \L)$.

Consider the $2$-Veronese image $S$ of $X$. Since
\[
h^0(E,\Sym^2(\O_E\oplus \L))=h^0(E,\O_E)+h^0(E,\L)+h^0(E,\L^{\otimes 2})=3g-2,
\]
the surface $S$ is linearly normally embedded in $\P^{3(g-1)}$. The
bicanonical image of $C$ is a hyperplane section of $S$, and it is
linearly normally embedded with degree $d=4(g-1)$.
In this case we will say that $S$ presents the
\emph{bicanonical bielliptic case}.
\end{example}

\begin{example}\label{ex:plane} Consider the linear system
$|\O_{\P^2}(\delta)|$ of plane curves of degree $\delta$ whose
self-intersection is $\delta^2$ and whose genus is
\[
g=\frac {(\delta-1)(\delta-2)}2.
\]
We assume $\delta\geq 4$ so that $g>1$. One has $\delta^2\geq 4g-4$ if and only
if $\delta\leq 6$. This means that for $4\leq \delta\leq 6$, the
degree of the $\delta$-Veronese image of $\P^2$
(and suitable simple internal projections of it)
is in the range $[4g-4,4g+4]$.

More precisely, in the case $\delta=4$. one has $\delta^2=4g+4$, which
is the maximum possible degree with respect to the sectional genus. We
still have degree in the above range if we make simple internal
projections of the $4$-Veronese of $\P^2$ from $b\leq 8$ points.

If $\delta=5$, we have $\delta^2=4g+1$, and the degree is in the above range
if we make simple internal projections of the $5$-Veronese of $\P^2$ from
$b\leq 5$ points.

Finally, if $\delta=6$,  we have $\delta^2=4g-4$.

In all these cases we will say that the surfaces present the
\emph{planar case}.
\end{example}

\begin{example}\label{ex:delp}%
\begin{a-enumerate}[wide]%
\item\label{e:d-1} Let $X$ be the plane blown up at $h\leq
7$ (proper or infinitely near) points such that there is an
irreducible cubic curve passing simply through these points.
Let $E_1,\ldots, E_h$ be the
exceptional $(-1)$-divisors over the blown-up points, set
$E=E_1+\cdots+E_h$, and let $H$ be the pull-back on $X$ of a general
line of~$\P^2$. Note that the anticanonical system on $X$ is
$|3H-E|$.
We will say that we are here in a \emph{Del Pezzo
  situation}
(even though $X$ is a genuine \emph{Del Pezzo surface} only if the
anticanonical system is ample).

Consider the linear system $|6H-2E|$. This linear system is
base-point-free, and its general curve is irreducible of genus
$g=10-h$ and self-intersection $4(9-h)=4g-4$. Moreover, it is not
difficult to see that $\phi_{|6H-2E|}$ is a birational morphism to the
image $S$, that is non-degenerate in $\P^{27-3h}=\P^{3g-3}$.  Note
that the hyperplane sections of these surfaces are bicanonically
embedded. For $h=0$, we again get the planar case for $\delta=6$.

\item\label{e:d-2} Similarly, let $X$ be an irreducible quadric in $\P^3$, and
consider the linear system $|-2K_X|$, which is the linear system of
quadric sections of $X$. Thus, either $X$ is the image of $\F_0$
by the linear system $|H+F|$ and $|-2K_X| =
|4H+4F|$ (curves of bidegree $(4,4)$ on $\P^1\times \P^1$), or
$X$ is the image of $\F_2$ by the linear system $|2H|$ and we may
identify $X$ with $\F_2$ and $|-2K_X|$ with $|4H|$.
As in case~\eqref{e:d-1}, the linear system $|-2K_X|$
is base-point-free, its general member is irreducible of genus $g=9$ and
self-intersection $4g-4=32$, and the associated map
$\phi_{|-2K_X|}$ is a birational morphism to the image $S$,
which is a non-degenerate surface in $\P^{3g-3}=\P^{24}$ with
bicanonical hyperplane sections.
\end{a-enumerate}
  We will say that the surfaces in cases~\eqref{e:d-1} and~\eqref{e:d-2} above present the \emph{bicanonical Del Pezzo case}.
\end{example}

\begin{example}\label{ex:hyperell} 
Consider, in the Notation~\ref{nz:scrolls}, the linear system 
\begin{equation*}
  |2H+kF|
  = |2E+(k+2e)F|
  = |H+E+(k+e)F|
\end{equation*}
on a rational ruled surface $\F_e$,
with $k\geq \max(0,3-e)$.
It
is base-point-free, and very ample unless $k=0$,
in which case the morphism $\phi_{|2H|}$ maps $\F_e$ birationally 
onto its image, which is the $2$-Veronese image of the cone in
$\P^{e+1}$ over a rational normal curve in $\P^{e}$. In any event the
general curve in $|2H+kF|$ is smooth and irreducible.

Since
\begin{equation}
\label{eq:kappa}
  K_{\F_e}
  \lineq -2E-(e+2)F
  \lineq -2H+(e-2)F,
\end{equation}
the adjoint system to $|2H+kF|$ is
$|(k+e-2)F|$, and therefore the curves in $|2H+kF|$ are
hyperelliptic of genus $g=k+e-1$.
The assumption that $k\geq 3-e$ implies that $g\geq 2$.
Moreover,
\[
(2H+kF)^2=4e+4k=4g+4.
\]
If $C$ is a smooth curve in $|2H+kF|$, then from the exact sequence
\[
0\longrightarrow \O_{\F_e}\longrightarrow \O_{\F_e}(2H+kF)\longrightarrow
\O_{C}(2H+kF)\longrightarrow 0
\]
and from the fact that
\[
h^1(\F_e, \O_{\F_e})=0, \quad h^0(C,\O_{C}(2H+kF))=3g+5, 
\]
we deduce that $h^0(\F_e,\O_{\F_e}(2H+kF))=3g+6$. If we set $S=
\phi_{|2H+kF|}(\F_e)$, then $S$ is non-degenerate of degree $4g+4$ in
$\P^{3g+5}$, and its general curve section is hyperelliptic of genus
$g\geq 2$.  Any surface which is a simple internal projection of $S$
from $b\leq 8$ points on $S$ still has degree in the range
$[4g-4,4g+4]$ and sectional genus $g$.  We will say that surfaces of
this type present the \emph{hyperelliptic case}.
\end{example}

\begin{example}\label{ex:trig} Let $C$ be a trigonal canonical curve
of genus $g\geq 4$ in $\P^{g-1}$. Then $C$ sits on a smooth rational
normal scroll $Y$ of degree $g-2$ in $\P^{g-1}$. We denote by $\cH$
the hyperplane section class of $Y \subset \P^{g-1}$ and by $F$ a line
of the ruling of $Y$ (be careful not to mistake $\cH$ for the class
$H$ in our Notation~\ref{nz:scrolls}; the minimal resolution of
singularities of $Y$ is isomorphic to some rational ruled surface
$\F_e$, and one has $\cH = H+lF$ for some $l\geq 0$).

It is easy to check that $C\in |3\cH-(g-4)F|$. Conversely, if $Y$
is a rational normal scroll of degree $g-2$ in $\P^{g-1}$, and
if a smooth curve $C$ sits in $|3\cH-(g-4)F|$, then $C$ is a trigonal
canonical curve. One has $(3\cH-(g-4)F)^2=3g+6$; hence the linear system
$|\O_C(3\cH-(g-4)F)|$ is very ample of dimension $2g+6$. This shows that
$\phi_{|3\cH-(g-4)F|}$ is a morphism that maps $Y$ birationally to a
non-degenerate surface $S \subset \P^{2g+6}$. If $g\leq 10$, one has
$3g+6\geq 4g-4$.
We will say that surfaces of this sort, as well as their simple
internal projections of degree at least $4g-4$,
present the \emph{trigonal case}.

To connect with the notation in Theorem~\ref{t:hideg}, note that if
$\cH = H+lF$ on $\F_e$, then
\[
  (H+lF)^2 = e+2l=g-2
\]
and
\[
  3\cH-(g-4)F
  = 3H+(3l-g+4)F
  = 3H + \tfrac{1}{2} (g-3e+2)F.
\]
\end{example}

\bigskip
Note that none of the surfaces in the above examples is a
cone. Indeed, they have irregularity $q\leq 1$ and sectional genus $g
\geq 2$.

\subsection{Previously known results}

We quote the following from \cite[Section~7]{angelo} but do not use it in
our proof.

\begin{theorem}[\cf \protect{\cite[Corollary~2.10]{kl}, \cite[Theorem~2]{bel}}]
Let $C \subset \P^r$ be a smooth irreducible non-degenerate linearly normal curve of genus $g \geq 4$ and degree $d$. Then $C$ is not extendable if
\begin{a-enumerate}
\item\label{t36-1}  $C$ is trigonal, $g \geq 5$, and $d \geq \max\{4g-6,3g+7\}$;
\item\label{t36-2}  $C$ is a plane quintic and $d \geq 26$;
\item\label{t36-3}  $\Cliff(C) = 2$ and $d \geq 4g-3$;
\item\label{it:gnl} $\Cliff(C) \geq 3$ and $d \geq 4g+1-3\Cliff(C)$.
\end{a-enumerate}
\end{theorem}

Part \eqref{it:gnl} tells us, in particular, that no curve $C$ with
$\Cliff(C)>2$ is extendable in the range of degree under consideration
in the present text, namely $d \geq 4g-4$. If $\Cliff(C) \leq 2$, our
Theorem~\ref{t:hideg} classifies those extensions that indeed exist in
the possibilities left open by the above statement.

If $\Cliff(C)=2$, the only possibility in our range left by the above
theorem is $d=4g-4$.  We find that there indeed exist extensions in
this degree, and they are all extensions of bicanonical curves, for
some special curves.

Items~\eqref{t36-1} and~\eqref{t36-2} deal with curves of Clifford
index~$1$. For plane quintics, the maximal degree is $25$, and it is
indeed realized in our classification, by rational surfaces
represented by a linear system of plane quintics. For smooth
quintics, $g=6$, hence $4g-4=20$ and $4g+4=28$.

For trigonal curves, the above theorem says that extensions may have
degree at most 
\[
  \max(4g-7, 3g+6) =
  \begin{cases}
    4g-7 & \text{if $g\geq 13$}, \\
    3g+6 & \text{if $g\leq 13$}.
  \end{cases}
\]
If $g \geq 13$, this implies that there is nothing in the range
$[4g-4,4g+4]$. For $g\leq 13$, the bound is sharp and, by our
classification, realized exclusively by trisecant scrolls (\cf
Example~\ref{ex:trig}) as soon as $3g+6 \geq 4g-4$, \ie $g\geq 10$.

For curves with Clifford index zero, \ie hyperelliptic curves, there
exist extensions in all degrees $4g-4,$ $\ldots,4g+4$, and they are all
bisecant scrolls (\cf Example~\ref{ex:hyperell}).

\subsection{Setup of the proof of Theorem~\ref{t:hideg}}
\label{s:setup}

We now start proving Theorem~\ref{t:hideg}.  Let $S \subset \P^{r+1}$
be a surface as in the theorem.  It may have at most isolated singular
points.  We consider its minimal desingularization $\pi\colon S'\to
S$, and, abusing notation, we still denote by $C$ its proper transform
on $S'$.  By the minimality of the resolution, there is no irreducible
$(-1)$-curve $E$ such that $C\cdot E=0$ on $S'$.

Since $K_{S'}\cdot C=2g-2-d<0$, the Kodaira dimension of $S'$ is
$\kappa(S')=-\infty$.  We let $q$ be the irregularity of~$S'$ and
consider a minimal model $f\colon S'\to \Sigma$ of $S'$. If $q>0$,
then $\Sigma$ is a $\P^1$-bundle over a smooth curve $\Gamma$ of genus
$q$, while if $S'$ is rational, then $\Sigma$ is either $\P^2$ or a
rational ruled surface $\F_e$, with $e\geq 0$ and $e\neq 1$.

\subsection{The irregular case}\label{sec:irreg}

In this section we will prove Theorem~\ref{t:hideg} in the case $q>0$,
which amounts to proving the following proposition.

\begin{proposition}\label{thm:irreg}
Let $S$ be as in Theorem~\ref{t:hideg} with $q>0$.  Then $S$ presents
the bicanonical bielliptic case as in Example~\ref{ex:biell}; hence we
are in case \eqref{ghd:biell} of Theorem~\ref{t:hideg}.
\end{proposition}

In this case the minimal desingularization $S'$ of $S$ has a
surjective morphism $\varphi\colon S'\to \Gamma$ to a smooth curve
$\Gamma$ of genus $q$ with connected, rational fibres.  We denote by
$\theta$ the class in the N{\'e}ron--Severi group of $S'$ of a general
fibre of $\varphi$ and set
\[
m=C\cdot \theta.
\]
This is the degree of the images of the fibres of $\varphi$ on $S$.
If $m=1$, then, by Theorem~\ref{l:CSegre}, $S$ is a cone, which is
excluded. Thus we have $m\geq 2$.  In the case $m=2$, the images of
the fibres of $\varphi$ on $S$ are conics, and we will say that we are
in the \emph{conic case}.

The proof of Proposition~\ref{thm:irreg} consists of a few steps.

\subsubsection{Reduction to the conic case}\label{ssec:red1}

The first step consists in the following. 

\begin{proposition}\label{prop:conic}
  Let $S$ be as in Proposition~\ref{thm:irreg}.
  Then $m=2$ and $d=4(g-1)$.
\end{proposition}

\begin{proof} We apply \cite[Theorem (2.3)]{hartshorne-bound} to the effect that if $C$ is a smooth, irreducible curve of genus $g$ on an irregular ruled surface $S'$  with $m>1$, then
\[
C^2\leq \frac {2m}{m-1}(g-1).
\]
Hence we have 
\[
4(g-1)\leq d=C^2\leq \frac {2m}{m-1}(g-1),
\]
and the assertion follows immediately. \end{proof}

\subsubsection{Passing to a minimal model}\label{ssec:red2}

Consider the image $\bar C$ of $C$ in $\Sigma$ via the map $f\colon S'
\to \Sigma$.  This is an irreducible curve, and since $m=2$ by
Proposition~\ref{prop:conic}, $\bar C$ may have at most double points
that can be proper or infinitely near. Let $h$ be the number of double
points of $\bar C$. One has
\[
\bar C^2=C^2+4h+\nu, \quad p_a(\bar C)=g+h,
\]
with $\nu$ the number of $(-1)$-curves $E$ contracted by $f$
and such that $C\cdot E = 1$.

By performing a sequence of elementary transformations on $\Sigma$
based at the double points of $\bar C$, we
produce a birational map
\[
\alpha\colon \Sigma \longdashrightarrow  \Sigma', 
\]
where $\Sigma'$ is still a $\P^1$-bundle  over $\Gamma$, the image
$C'$ of $\bar C$ via $\alpha$ is smooth, and
$C'^2=C^2+\nu=4(g-1)+\nu$.
Applying Proposition~\ref{prop:conic} to $C' \subset \Sigma'$,
one finds that $\nu=0$.

Abusing notation, we will still denote by $\theta$ the class of a fibre of the structure morphism $\varphi'\colon \Sigma'\to \Gamma$.

The pair $(S',C)$ is birational to the pair $(\Sigma',C')$; \ie there is a birational map $S'\dasharrow \Sigma'$ that maps $C$ to $C'$. Hence the image of $\Sigma'$ via the map $\phi_{|C'|}$ determined by the linear system $|C'|$ is the original surface $S$. So, rather than studying the linear system $|C|$ on $S'$, we may study the linear system  $|C'|$ on $\Sigma'$.

\subsubsection{The adjoint system}\label {ssec:adj}

Next we consider the \emph{adjoint linear system} $|K_{\Sigma'}+C'|$. 

Let $\restr {\varphi'}{C'}\colon C'\to \Gamma$ be the double cover, with
branch divisor $B$. By the Riemann--Hurwitz formula, 
we have
\[
g=2q-1+b, \quad \text{with}\ \deg(B)=2b.
\]

From the cohomology sequence of the exact sequence
\[
0\longrightarrow  \O_{\Sigma'}(K_{\Sigma'})\longrightarrow  \O_{\Sigma'}(K_{\Sigma'}+C')\longrightarrow  \O_{C'}(K_{C'})\longrightarrow 0,
\]
since 
\begin{align*}
&h^0(\Sigma', \O_{\Sigma'}(K_{\Sigma'}))=h^1(\Sigma',  \O_{\Sigma'}(K_{\Sigma'}+C'))=0, \\
          &h^1(\Sigma', \O_{\Sigma'}(K_{\Sigma'}))=q,  \\
          &h^0(C',  \O_{C'}(K_{C'}))=g,
          \end{align*}
we get
\[
\dim(|K_{\Sigma'}+C'|)=g-1-q=q-2+b.
\]
Moreover, $q-2+b\geq 0$ because we are assuming $q>0$ and $g\geq 2$. 

Note that 
$K_{\Sigma'}\cdot \theta=-2$, hence $(K_{\Sigma'}+C')\cdot \theta=0$,
so, since $K_{\Sigma'}+C'$ is effective,
we have the numerical equivalence $K_{\Sigma'}+C'\equiv k\theta$
for some integer $k$.

\begin{lemma}\label{lem:crx} In the above setting one has $q=1$. 
\end{lemma}

\begin{proof} We have 
\[
  0=(K_{\Sigma'}+C')^2
= K_{\Sigma'}^2 +2(K_{\Sigma'}+C')\cdot C' -(C')^2
=K_{\Sigma'}^2+4(g-1)-C'^2
=K_{\Sigma'}^2.
\]
On the other hand, $K_{\Sigma'}^2=8(1-q)$, and the assertion follows.
\end{proof}

\subsubsection{The classification}\label{ssec:class}

We may now finish the proof of Proposition~\ref{thm:irreg}. To do so,
we first identify the surface $\Sigma'$. We have $\Sigma'=\P(\mathcal
E)$, where $\mathcal E$ is a rank $2$ vector bundle on a curve
$\Gamma$ of genus 1. We can suppose that $\mathcal E$ is
\emph{normalized} (see \cite[Notation~V.2.8.1]{hartshorne}),
with invariant
$e$. We denote by $E$ the section such that $E^2=-e$. We have
\[
K_{\Sigma'}\equiv
-2E-e\theta
\]
by \cite[Corollary~V.2.11]{hartshorne} 
and
\begin{equation}\label{eq:ccc}
C'\equiv 2E+a\theta
\end{equation}
for some integer $a$.

One has
\begin{align*}
  2g-2=(K_{\Sigma'}+C')\cdot C'
= (a-e)\theta\cdot C'
    = 2(a-e),
\end{align*}
hence 
\begin{equation}\label{eq:aaa}
  a=g-1+e.
\end{equation}

\begin{lemma}\label{lem:dec} The vector bundle $\mathcal E$ is decomposable.
\end{lemma}

\begin{proof} Suppose towards a contradiction that $\mathcal E$ is
  indecomposable. Then one has $-1\leq e\leq 0$ by \cite[Theorem~V.2.15]{hartshorne}. If $e=0$, then by  \eqref {eq:aaa}, we have
\begin{equation}\label{eq:bbb}
C'-K_{\Sigma'}\equiv 4E+(g-1)\theta,
\end{equation}
which is big and nef. Then, by Kawamata--Viehweg vanishing, one has $h^1(\Sigma', \O_{\Sigma'}(C'))=0$. But then from the cohomology sequence of the exact sequence
\[
0\longrightarrow \O_{\Sigma'}\longrightarrow  \O_{\Sigma'}(C')\longrightarrow  \O_{C'}(C')\longrightarrow 0,
\]
we see that the restriction map
\[
H^0(\Sigma',\O_{\Sigma'}(C') )\longrightarrow H^0(\Sigma',\O_{C'}(C'))
\]
has corank $h^1(\Sigma',\O_{\Sigma'})=1$, so it is not surjective, which in turn implies that $C$ is not linearly normal, so we have a contradiction.

In the case $e=-1$ one has
$h^1(\Sigma', \O_{\Sigma'}(C'))=0$
by \cite[Theorem~1.17]{CaCi}, and then one concludes as in the previous
case. 
\end{proof}

We can now finish the proof. 

\begin{proof}[Proof of Proposition~\ref{thm:irreg}]
From Lemma~\ref{lem:dec}, we have that
\[
\mathcal E=\O_{\Gamma}\oplus \L, 
\]
where $\L$ is a line bundle of degree $-e\leq 0$, so that $e\geq 0$ (see \cite[Theorem V.2.12]{hartshorne}). We want to compute $h^1(\Sigma', \O_{\Sigma'}(C'))$. To do so, consider again the structure morphism $\varphi'\colon \Sigma'=\P(\mathcal E)\to \Gamma$, and note that 
\begin{align*}
  h^1(\Sigma', \O_{\Sigma'}(C'))
  = h^1(\Gamma, \varphi'_* \O_{\Sigma'}(C')).
\end{align*}
From \eqref {eq:ccc} and \eqref {eq:aaa}, we have
\begin{equation}
\label{eq:forzaNapoli}
\varphi'_* \O_{\Sigma'}(C')=\Sym^2(\mathcal E)\otimes \mathcal D
\end{equation}
where $\mathcal D$ is a line bundle of degree $g-1+e$ on $\Gamma$. One has
\[
\Sym^2(\mathcal E)=\O_\Gamma\oplus \L\oplus \L^{\otimes 2},
\]
hence
\[
\Sym^2(\mathcal E)\otimes \mathcal D=\mathcal D \oplus (\mathcal D\otimes \L)\oplus (\mathcal D \otimes \L^{\otimes 2})
\]
with 
\[
\deg (\mathcal D\otimes \L)=g-1, \quad \deg (\mathcal D \otimes \L^{\otimes 2})=g-1-e.
\]
Therefore,
\[
h^1(\Sigma', \O_{\Sigma'}(C'))=h^1(\Gamma, \mathcal D)+h^1(\Gamma,\mathcal D\otimes \L)+h^1(\Gamma, \mathcal D \otimes \L^{\otimes 2})=h^1(\Gamma, \mathcal D \otimes \L^{\otimes 2}).
\]
By the same argument we made in the proof of Lemma~\ref{lem:dec}, we must have 
$h^1(\Sigma', \O_{\Sigma'}(C'))>0$. So we must have $h^1(\Gamma, \mathcal D \otimes \L^{\otimes 2})>0$, hence $ g-1-e =\deg (\mathcal D \otimes \L^{\otimes 2})\leq 0$ and $e\geq g-1$. On the other hand, by \eqref {eq:ccc} and \eqref {eq:aaa},  we have
\[
C'\cdot E=(2E+(g-1+e)\theta)\cdot E=g-1-e,
\]
hence $e\leq g-1$, thus $e=g-1$, and from $h^1(\Gamma, \mathcal D \otimes \L^{\otimes 2})>0$, we deduce that
\[
\mathcal D \otimes \L^{\otimes 2}\cong \O_\Gamma, \quad \ie\;\; \mathcal D\cong 
{(\L^\vee)}^{\otimes 2}.
\]

On the other hand, let us compute the linear equivalence class of
$C'$. By \eqref {eq:ccc},
\[
  \O_{\Sigma'}(C')=\O_{\Sigma'}(2E)\otimes \varphi'^* \mathcal{M}
\]
for some line bundle $\mathcal{M}$ on $\Gamma$,
hence
$\varphi'_* \O_{\Sigma'}(C')=\Sym^2 \mathcal{E} \otimes \mathcal{M}$.
Then \eqref{eq:forzaNapoli} implies that
$\mathcal{M} = \mathcal{D}$.
The upshot is that
\[
\O_{\Sigma'}(C')=\O_{\Sigma'}(2E)\otimes \varphi'^*({(\L^\vee)}^{\otimes 2})=\mathcal A^{\otimes 2}, 
\]
where we set
\[
\mathcal A=\O_{\Sigma'}(E)\otimes \varphi'^*{(\L^\vee)}.
\]
The map $\phi_\mathcal A$ determined by the line bundle $\mathcal
A$ maps $\Sigma'=\P(\O \oplus \mathcal L)$ to a cone $X\subset \P^{g-1}$ over
the elliptic normal  curve of degree $g-1$  in $\P^{g-2}$ which is the
image of $\Gamma$ via the map $\phi_{\L^\vee}$. In this map the
curve $C'$ is mapped to a quadratic section of $X$. This implies that
we are in the bicanonical bielliptic case.
\end{proof}

\subsection{The rational case}\label{sec:rat}

In this subsection we finish the proof of Theorem~\ref{t:hideg}
by considering the case in which the surface $S$ is rational.
We will thus prove the following. 

\begin{proposition}\label{thm:rat} 
  Let $S$ be as in Theorem~\ref{t:hideg} with $q=0$; \ie $S$ is
  rational.  Then $S$ presents either the planar case $($see
  Example~\ref{ex:plane}\,$)$ or the bicanonical Del Pezzo case $($see
  Example~\ref{ex:delp}\,$)$ or the hyperelliptic case $($see
  Example~\ref{ex:hyperell}\,$)$ or the trigonal case $($see
  Example~\ref{ex:trig}\,$)$.
\end{proposition}
 
The proof will consist of various steps that we will carry out in the
next subsections. 
We consider the adjoint system 
$|C+K_{S'}|$, which we write
\[
|C+K_{S'}|=F+|M|,
\]
where $F$ is the fixed part and $|M|$ the movable part, with $\dim (|M|)=g-1$. Since $|C+K_{S'}|$  cuts the complete canonical linear series on $C$, we have $C\cdot F=0$. 

There are two cases to be considered:
\begin{a-enumerate}
\item  \label{case:he}  
 $|M|$ is composed with a pencil $|\Phi|$,
including the case $g=2$ in which $|M|$ itself is a pencil.
\item \label{case:nonhe}
The general curve in $|M|$ is irreducible, and
$\dim(|M|)\geq 2$, hence $g\geq 3$.
\end{a-enumerate}

In case \eqref{case:he} the curves in $|C|$ are hyperelliptic and
$C\cdot \Phi=2$.
So the curves in $|\Phi|$ are mapped to conics on $S \subset \P^{r+1}$
and therefore are rational.

\subsubsection {The hyperelliptic case}
\label{ghd:hell}

\begin{proof}%
[Proof of Proposition~\ref{thm:rat}  in case \eqref{case:he}]  
There is a birational morphism $\xi\colon S'\to \F_e$ such that the pencil
$|\Phi|$ of rational curves is mapped to the system $|F|$ of fibres of
the structure morphism $\F_e\to \P^1$. Then the linear system  $|C|$
is mapped to a linear system of curves of  type $|2H+kF|$
and, by acting if necessary with elementary transformations, we may assume
that the general curve in this system is smooth of genus $g$.  The
adjoint system to $|2H+kF|$ is $|(k+e-2)F|$,
and therefore $g=k+e-1$. Since  $g\geq 2$, we must have $k\geq
3-e$. Moreover, since $k=(2H+kF)\cdot E\geq 0$, we must also have
$k\geq 0$. Thus we are in the hyperelliptic case.
\end{proof}

\subsubsection{The non-hyperelliptic case}

Next we consider case \eqref{case:nonhe}.
In particular, $g\geq 3$.

\begin{lemma}\label{lem:nonhyp} In case \eqref{case:nonhe}
  the general curve in $|C|$ is not hyperelliptic.
\end{lemma}

\begin{proof} Suppose towards a contradiction that the general curve in $|C|$
  is hyperelliptic. 
Let $p$ be a general point of $S'$. Let us consider a general pencil
$\mathcal P$ of curves in $|C|$ having $p$ as base point. If $C$ is a
general curve in $\mathcal P$ and $q$ is the point conjugate to $p$ in
the $g^1_2$ on $C$, then by the generality of $\mathcal P$, we can
assume that $q$ is not a base point of $\mathcal P$.  Consider the
Zariski closure $D$ in $S'$ of the set of points $q\in S'$ such that
$p+q$ is a divisor of the $g^1_2$ on the curves of $\mathcal P$; 
\ie
\[
  D =\overline
  {
    \bigcup _{C \in \mathcal{P}^0}
    \left\{ q: p+q \in g^1_2 (C) \right\},
  }
\]
with $\mathcal P^0 \subset \mathcal P$ the Zariski dense open subset
parametrizing smooth members of $\mathcal{P}$.
Then $D$ is a (rational) curve on $S'$.

Any curve of $|M|$ passing through $p$ (which, by the generality of
$p$, is a general curve of $|M|$) contains $D$ and therefore
coincides with $D$.  On the other hand, we claim that $D$ does not cut
out a canonical divisor on~$C$; hence it cannot be a member of $|M|$,
and we have a contradiction.

It thus only remains to prove the claim.
Let $m$ be the multiplicity of
$p$ in $D$. 
Then
\[
  D \cdot C = mp+q+R,
\]
where $R$ is contained in $\mathrm{Bs}(\mathcal{P})-p$, with
$\mathrm{Bs}(\mathcal{P})$ the base locus of $\mathcal{P}$.  If $R=0$,
the claim holds (recall that $g\geq 3$).  Otherwise, $R$ must contain
all $d-1$ points of $\mathrm{Bs}(\mathcal{P})-p$ since by the
generality of $\mathcal{P}$, there is a monodromy action on
$\mathrm{Bs}(\mathcal{P})-p$, and it acts as the full symmetric group
by \cite[Section~III.1, pp.~111--113]{ACGH}.
Then the claim follows as $d \geq 4g-4$.
\end{proof}

One has $C\cdot M=2g-2$.
Hence for an irreducible curve $M$ to be contained in a curve in $|C|$
is at most $2g-1$ conditions, and equality holds if and only if $M$ is
smooth and rational and the restriction of $|C|$ to $M$ is a complete
linear series. We thus have
\[
\dim (|C-M|)\geq 3g-3+\varepsilon-(2g-1)=g-2+\varepsilon\geq 1,
\]
where $\epsilon$ is the non-negative integer such that
$C^2 = 4g-4+\epsilon$.

\begin{proof}[Proof of Proposition~\ref{thm:rat}  in the case $\dim
  (|C-M|)=g-2+\varepsilon$]
In this case the general curve in $|M|$ is smooth and rational. Since $\dim (|M|)=g-1$, we have $M^2=g-2$, $|M|$ is base-point-free, and $\phi_{|M|}$ is a morphism mapping $S'$ to a surface $Y\subseteq \P^{g-1}$. 
In this map the curves $C$ are mapped to canonical curves of degree
$2g-2$. Since $Y$ has rational hyperplane sections, we have only the
following possibilities: 
\begin{a-enumerate}
\item\label{pft:r-a}  $g=3$, and $Y$ is $\P^2$.
\item\label{pft:r-b}  $g=6$, and $Y$ is the $2$-Veronese image of $\P^2$.
\item\label{pft:r-c}  $Y$ is a rational normal scroll.
\end{a-enumerate}

In cases~\eqref{pft:r-a} and~\eqref{pft:r-b} we are in the planar case
with $\delta=4$ and $\delta=5$, respectively. In case~\eqref{pft:r-c}
we are in the trigonal case.
\end{proof}

Next we assume $s:=\dim (|C-M|) \geq g-1+\varepsilon$.
Recall that 
\[
C-M \sim -K_{S'}+F.
\]

\begin{lemma}\label{lem:fixo}
In the above setting, $F$ is in the fixed part of\, $|C-M|$.
\end{lemma}

\begin{proof}
Let $D$ be an irreducible component of $F$. One has $C\cdot D=0$,
hence $D^2<0$. Then $K_{S'}\cdot D\geq 0$; otherwise, $D$ would be a
$(-1)$-curve contracted by $|C|$, and we would have a
contradiction. So we have $K_{S'}\cdot F\geq 0$ and therefore
$(-K_{S'}+F)\cdot F<0$.  Indeed, we also have $F^2<0$ by the Hodge
index theorem because $C\cdot F=0$.  Since $-K_{S'}+F\sim C-M$ is
effective, there is a non-zero divisor $G\leq F$ that is in the fixed
part of $|-K_{S'}+F|$. If $G=F$, we are done. Otherwise, set $F_1=F-G$
and consider the linear system $|-K_{S'}+F_1|$. By the same argument
as above, we have $(-K_{S'}+F_1)\cdot F_1<0$, so there is a non-zero
divisor $G_1\leq F_1$ that is in the fixed part of $|-K_{S'}+F_1|$. If
$G_1=F_1$, we are done. Otherwise, we repeat this argument till we
eliminate all of $F$ from the fixed part of $|-K_{S'}+F|$.
\end{proof}

\begin{proof}[Conclusion of the proof of Proposition~\ref{thm:rat}]
By Lemma~\ref{lem:fixo}, we can assume that $s=\dim (|-K_{S'}|)\geq g-1+\varepsilon\geq 2$. Consider the map
\[
\phi_{|-K_{S'}|}\colon S'\longdashrightarrow Y\subseteq \P^s.
\]
One has 
\[
(-K_{S'})\cdot C=(C-M-F)\cdot C=4g-4+\varepsilon -(2g-2)=2g-2+\varepsilon.
\]
Hence the curves in $|C|$ are mapped via $\phi_{|-K_{S'}|}$ to curves
of degree $\delta\leq 2g-2+\varepsilon$. Let $\gamma$ be the number of
conditions that containing the curve $C$ imposes on the members of
$|-K_{S'}|$. One has
\begin{equation*}
    \begin{cases}
    \gamma =g & \text{if }\,\,\,  \varepsilon =0\,\,\,   \text {and}\,\,\,
    C\,\,\,  \text{is mapped via} \,\,\,  \phi_{|-K_{S'}|}\,\,\,
    \text{to a canonical curve},  \\
    \gamma \leq g-1+\varepsilon & \text{otherwise}.
  \end{cases}
\end{equation*}
In any event, unless $\varepsilon =0$, $s=g-1$, and $\gamma=g$, one has
\[
\dim (|-K_{S'}-C|) = s-\gamma\geq  0, 
\]
so that $-K_{S'}-C$ is effective. Hence we have $-K_{S'}\sim C+T$, with $T$ effective. Then 
$0\sim K_{S'}+ C+T\sim F+M+T$, which is not possible. 

So the only possibility is that $\varepsilon =0$, $s=g-1$, $\gamma=g$,
and $|-K_{S'}|$ cuts out the complete canonical series on $C$. Since
$\dim(|-K_{S'}|)\leq 9$, we have $g\leq 10$. 

From $\restr {-K_{S'}} C =K_C$ we deduce that $\O_{C}(C)=\O_{C}(2K_C)=\O_C(-2K_{S'})$.

Let us set $|-K_{S'}|=A+|B|$, where $A$ is the fixed part and $|B|$
the movable part of $|-K_{S'}|$. Since $|-K_{S'}|$ cuts out the
complete canonical series on $C$ and $C$ is not hyperelliptic by
Lemma~\ref{lem:nonhyp}, we have that $C\cdot A=0$ and $|B|$ is not
composed with a pencil; hence the general curve in $|B|$ is
irreducible.

First suppose that $A=0$. Then, for any irreducible curve $B\in |B|$, one has $p_a(B)=1$. Moreover, from the exact sequence
\[
0\longrightarrow \O_{S'}\longrightarrow \O_{S'}(B)\longrightarrow \O_{B}(B)\longrightarrow 0
\] 
and from $h^0(S',\O_{S'}(B))=g$, we deduce that
$h^0(B,\O_{B}(B))=g-1$, which implies that $B^2=g-1$.  From the index
theorem applied to $C$ and $B$, we have that $C\sim 2B$.  Moreover,
there is a birational map $\eta\colon S'\dasharrow \P^2$ that maps $|B|$ to
the linear system of cubics with $10-g$ simple base points. From this,
we see that we are in the bicanonical Del Pezzo case.

Next suppose that $A$ is non-zero.  By Lemma~\ref{l:antican-rat}
below, $A\cdot B=2$ and the general member of $|B|$ is rational.  By
the same argument we made above, we have $B^2=g-2$, and
$\phi_{|B|}=\phi_{|-K_{S'}|}$ is a birational map from $S'$ to its image
$Y$ that is a surface of minimal degree in $\P^{g-1}$.

If $Y=\P^2$, then $\phi_{|B|}$ maps $|B|$ to the linear system of
lines and $A$ to a conic $\Gamma$. The curves in $|C|$ are mapped to
plane quartics. Since $C\cdot A=0$, the linear system $|C|$ is mapped
to a linear system of quartics with eight (proper or infinitely near)
base points on the conic $\Gamma$. This shows that we are in the
planar case with $\delta=4$ and $b=8$ (see the notation in
Example~\ref{ex:plane}).

If $Y$ is the $2$-Veronese image of $\P^2$, then we may identify $Y$
with $\P^2$ and $\phi_{|B|}$ maps $|B|$ to the linear system of conics
and $A$ to a line $R$. The curves in $|C|$ are mapped to plane
quintics. Since $C\cdot A=0$, the linear system $|C|$ is mapped to a
linear system of quintics with five (proper or infinitely near) base
points along the line $R$. Hence we are in the planar case with
$\delta=5$ and $b=5$ (see Example~\ref{ex:plane} again). 

Finally, if $g\geq 4$, $Y$ can be a rational normal scroll in
$\P^{g-1}$ and $\phi_{|B|}$ maps $|B|$ to the linear system $|\cH|$,
where $\cH$ is the hyperplane section class on $Y$. The curves in $|C|$
are mapped to canonical trigonal curves $C'$ on $Y$.
To see that $C'$ is trigonal, consider a divisor $D$ cut out on $C'$
by a line of the ruling of $Y$, and let $d$ be its degree; by the
geometric Riemann--Roch theorem, the linear series $|D|$ has dimension
$d-2$, and then it follows from Clifford's theorem that it is a
$g^1_3$ as $C'$ is non-hyperelliptic.
Let $F$ be a line
of the ruling of $Y$. As in Example~\ref{ex:trig}, one sees that
$C'\in |3\cH-(g-4)F|$, so that $Y$ is smooth, unless maybe $g=4$. In
the latter case, $Y$ can be a quadric cone, so its minimal
desingularization is $\F_2$, and in this case we will work on $\F_2$
rather than on $Y$.

Suppose that $Y=\F_e$, and let, as usual, $E$ be the section such that $E^2=-e$. Then $g$ and $e$ have the same parity and 
\[
\cH\sim E+\left( \frac {e+g}2-1\right)F
\]
and, accordingly, 
\[
C'\sim 3\cH-(g-4)F\sim 3E+\left( \frac {3e+g}2+1\right) F.
\]

The morphism $\phi_{|B|}\colon S'\to Y$ consists in blowing down 
a number of $(-1)$-exceptional divisors. Let $D$ be the total such exceptional divisor. Then
\[
-K_{S'}=-\phi_{|B|}^*(K_{\F_e})-D.
\]
Since $\phi_{|B|}^*(\cH)=B$, we have
\[
A\sim \phi_{|B|}^*(-K_{\F_e}-\cH)-D.
\]
By \eqref {eq:kappa}, we have
\[
-K_{\F_e}-\cH\sim E+\left(\frac {e-g}2+3  \right)F,
\]
thus
\[
A\sim \phi_{|B|}^*\left(E+\left(\frac {e-g}2+3  \right)F\right)-D.
\]

One has
\[
C'\cdot \left(E+\left(\frac {e-g}2+3  \right)F\right)=10-g; 
\]
recall that $C\cdot A=0$. Moreover, 
\[
-K_{\F_e}\cdot \left(E+\left(\frac {e-g}2+3  \right)F\right)=8-g<10-g.
\]
In conclusion:
\begin{itemize}
\item $S'$ is obtained from $\F_e$ by blowing up the (curvilinear) scheme $Z$ of length $10-g$ that is the complete intersection of a smooth curve $C'$ with a curve $N$ of $\big|E+(\frac {e-g}2+3)F\big|$.
\item The linear system $|C|$ on $S'$ is the strict transform of the linear systems of the curves of $|C'|$ on $\F_e$ containing $Z$, and
\[
C^2=C'^2-(10-g)=3g+6-(10-g)=4g-4,
\]
as we wanted.
\item The strict transform $A$ of $N$ on $S'$ splits off the
  anticanonical system on $S'$. 
\end{itemize}
Thus we are here in the trigonal case.
\end{proof}

This concludes the proof of Proposition~\ref{thm:rat}, hence also that
of Theorem~\ref{t:hideg}.
We end this section with an elementary lemma that has been used above.

\begin{lemma}
\label{l:antican-rat}
Let $S$ be a smooth rational surface, and write
\[
  |-K_S| = A + |B|
\]
with $A$ the fixed part and $|B|$ the movable part. We assume that
$A$ is effective and non-zero and $B$ big and nef.
Then $A\cdot B=2$, and the general member of\, $|B|$ is
rational if it is irreducible.
\end{lemma}

\begin{proof}
First, $h^0(-A)=0$.
Next,  by Kawamata--Viehweg vanishing, $h^1(-A)=h^1(B+K_S)=0$.
Finally,
\[
  h^2(-A) = h^0(K_S+A)= h^0(-B)=0.
\]
Thus $\chi(-A)=0$, hence by Riemann--Roch $A\cdot B = 2$.
Therefore, $(B+K_S)\cdot B = -A\cdot B = -2$,
and the result follows.
\end{proof}

\section{Classification of surfaces with a
  hyperelliptic section}
\label{S:class-hyperell}

This section is dedicated to Theorem~\ref{t:hyperell-notrat}.  We
prove it in Sections~\ref{s:he-gnl} and~\ref{s:he-sporad}  and
formulate it in useful alternative ways in Section
\ref{s:castelnuovo}.

We consider a minimal resolution of singularities $S' \to S$ and work
on $S'$.  Being a hyperplane section of $S$, $C$ must be contained in
the smooth locus of $S$, so that considered in $S'$ it cannot
intersect any curve contracted by $S' \to S$. It follows that,
considered in $S'$, $C$ is a big and nef divisor such that $C^2 = d$.

We want to apply the Reider--Beltrametti--Sommese theorem below to $C$
for $k=1$, which requires $C^2 \geq 9$.  We thus split the proof of
Theorem~\ref{t:hyperell-notrat} in two: the \emph{general case} in
which we assume that $d \geq 9$, equivalently either $g\geq 3$, or
$g=2$ and $d \neq 7, 8$, and the \emph{sporadic cases} in which $g=2$
and $d=7$ or $8$.

\subsection{Proof of Theorem~\ref{t:hyperell-notrat} in the  general case}
\label{s:he-gnl}

\begin{theorem}[Reider, \cf \cite{reider},
Beltrametti--Sommese, \cf \cite{beltrametti-sommese}]
\label{t:R/B-S}
Let $L$ be a nef line bundle on a smooth surface $S$ and $k$ be a
positive integer.  Assume that $L^2 \geq 4k+5$ and there exists a
$0$-dimensional subscheme $Z \subset S$ of length $k+1$ such that the
restriction
\[
  H^0(S,K_S+L) \lra H^0\left(Z, \restr [K_S+L] Z\right)
\]
is not surjective.  Then there exists an effective divisor $D$
containing $Z$ and such that
\[
  L\cdot D - k-1 \leq D^2 < k+1.
\]
\end{theorem}

\begin{proof}[Proof of Theorem~\ref{t:hyperell-notrat} in the
  general case]
  The general case means that we assume $d \geq 9$; see above.
  
For all divisors $x_1+x_2$ in the $g^1_2$ of $C$, the adjoint system
$|K_{S'}+C|$ does not 
separate $x_1$ and $x_2$, so Theorem~\ref{t:R/B-S}
tells us that there exists a divisor $D$ on $S'$ such that
\[
  x_1,x_2 \in D
  \quad \text{and}
  \quad
  D\cdot C -2 \leq D^2 \leq 1.
\]
The inequality $D \cdot C \leq 3$ forbids that $D$ contains $C$:
indeed, if $D=C+D'$ with $D'$ effective, then $D \cdot C = C^2+C\cdot
D' \geq C^2 \geq 9$, which contradicts $D\cdot C \leq 3$.
This implies that $D$ cannot be fixed as $x_1+x_2$ moves in the
$g^1_2$, for otherwise $D$ would necessarily have $C$ as an
irreducible component.

Let $M$ be the part of $D$ that contains $x_1$ and $x_2$.
By the previous observation, the family of these $M$ has no
fixed part. Thus we have
\[
  2 \leq M\cdot C \leq D\cdot C \leq 3.
\]
We first claim that $M$ is irreducible. Indeed, otherwise
at least one component $M'$ of $M$ would verify $M'\cdot C=1$; 
hence it would be mapped to a line by the map $S'\to S$ 
and give a ruling of $S$ by lines, which implies that
$S$ is a cone by Theorem~\ref{l:CSegre}.
So $M$ is irreducible, and we have the following possibilities \eqref{case1} and
\eqref{case2} to consider: 
\begin{enumerate}[label=(\arabic*),ref=\arabic*]
\item\label{case1} $M\cdot C=2$. Then $M$ is mapped to a conic by $S' \to S$; 
in particular, it is a rational curve.
Since it moves in a family
parametrized by the $g^1_2$ on $C$, the surface $S$ is then
unirational,  hence rational, so the conclusion of our theorem holds in
this case.

\item\label{case2} $M\cdot C=3$. Then $M$ is mapped to a cubic by $S' \to S$; 
  hence it spans either \begin{enumerate*}[label=(\alph*),ref=\alph*] \item\label{a} a plane or \item\label{b} a $3$-space in $\P^{d-g+1}$.\end{enumerate*}
  \end{enumerate}

Case~\eqref{case2} may happen only if $d\leq 9$, for the following
reason.  If $D\cdot C=3$, then $D^2=1$, and thus by the Hodge index
theorem,
\[
  D^2\cdot C^2 \leq (D\cdot C)^2
  \Iff
  C^2 \leq 9.
\]
The upshot is that $d=9$ and $g=2$ or $3$; hence $C$ is an intersection of
quadrics by Green's theorem, Theorem~\ref{t:green}.

In case~\eqref{a}, the three points of $M\cap C$ are located on the
line $\vect M \cap \vect C$. To see that the latter is indeed a line,
note that $\vect C$ is a hyperplane in $\vect S$; hence it is
impossible that the plane $\vect M$ be contained in $\vect C$ for
general $C$.  On the other hand, $C$ is an intersection of quadrics,
so this situation is impossible.  In conclusion, case~\eqref{a} cannot happen.

In case~\eqref{b} the curve $M$ is rational.  Hence if $M^2=0$, then
the curves $M$ are parametrized by a curve $\H_M$.  The curve $\H_M$
is rational because there exists a morphism $\P^1 \to \H_M$ mapping
the element $x_1+x_2$ of the $g^1_2$ to the corresponding curve $M$.
The upshot is that in this case $|M|$ is a base-point-free pencil.
The restriction of such a pencil to $C$ would be a base-point-free
$g^1_3$ (remember that $C$ viewed on $S'$ does not intersect any curve
contracted by $S' \to S$) containing all divisors of the $g^1_2$; by
Lemma~\ref{l:g12g13} below, this is impossible.  So we must have
$M^2>0$. But in this case $|M|$ cuts out a $g^r_3$ on $C$ with $r\geq
2$, in contradiction with $g \geq 2$.

The conclusion is that only case~\eqref{case1} may happen, in which
the conclusion of our theorem holds.
\end{proof}

\begin{lemma}
\label{l:g12g13}
  Let $C$ be a curve of genus $g$, and assume it has a $g^1_2$,
  $\mathfrak l$, and a $g^1_3$, $\mathfrak m$.
  We consider the following condition:
  \begin{equation}
    \label{cond-g12g13}
     \forall x_1+x_2 \in \mathfrak l,\
     \exists z \in C:\quad
     x_1+x_2+z \in \mathfrak m.
  \end{equation}
  If \eqref{cond-g12g13} holds,
  then either $z$ is a base point of\, $\mathfrak m$, or $g=0$.
  If \eqref{cond-g12g13} does not hold, then $C$ is birational to a
  curve of type $(2,3)$ in $\P^1\times\P^1$; in particular, 
  $g \leq 2$.
\end{lemma}

\begin{proof}
First assume that \eqref{cond-g12g13} holds.
Consider $z,z'\in C$ such that
$x_1+x_2+z \in \mathfrak m$
and $x'_1+x'_2+z' \in \mathfrak m$
for some members
$x_1+x_2$ and $x'_1+x'_2$ of $\mathfrak l$.
Then
\[
  x_1+x_2+z
  \lineq x'_1+x'_2+z',
  \quad \text{hence}
  \quad
  z \lineq z',
\]
and either $C$ is rational, or $z=z'$. In the latter case $z$ is a base
point of $\mathfrak{m}$.
This
proves the first part of the lemma.

For the second part, consider the map
\[
  \phi = (\phi_{\mathfrak l}, \phi_{\mathfrak m})\colon C
  \lra \P^1\times\P^1
\]
defined by the two pencils $\mathfrak l$ and $\mathfrak m$.
It is straightforward to verify that \eqref{cond-g12g13} holds if and only
if $\phi$ is not birational, and this proves the second part of the
statement. 
\end{proof}

\subsection{Proof of Theorem~\ref{t:hyperell-notrat} in the
  sporadic cases}
\label{s:he-sporad}

\begin{proof}[Proof of Theorem~\ref{t:hyperell-notrat} in the
  sporadic cases]
The sporadic cases are those when $g=2$ and $d=7$ or $8$.
Let $S \subset \P^{d-1}$, with $d = 7,8$, be a degree $d$ surface with one
hyperplane section a genus $2$ curve $C$, such that $S$ is not a cone,
hence not a scroll either by Theorem~\ref{l:CSegre}.
Since $d>2g-2$, we have $\kappa(S)=-\infty$ as in
Section~\ref{s:setup}. 

Assume towards a contradiction that $S$ is not rational.
Then $S$ is abstractly a ruled surface,
necessarily elliptic as it is not a scroll. Let $S'$ be a
minimal resolution of the singularities of $S$ and $\Sigma$ a minimal
model of $S'$; then $\Sigma$ is a ruled surface over an elliptic curve $B$.
By \cite[Theorem~1.4]{CC-wdef}, since $d\geq 7$,
if $S \subset \P^{d-1}$ is not $1$-weakly
defective, then for general $p_1,p_2 \in S$, the general section of $S$
by a hyperplane tangent at both $p_1$ and $p_2$ is a curve $C_0$
with two nodes at $p_1$ and $p_2$ and no other singularity. 

If $C_0$ is irreducible, then it is a rational curve
because $S$ has sectional genus $2$. Moreover, $C_0$ dominates the
base $B$ since $C_0^2>0$, so we have 
a contradiction to the non-rationality of $S$. 
Otherwise, there are the  following two possibilities:
\begin{enumerate}[label=(\arabic*),ref=\arabic*]
\item\label{41-1} $C_0 = \bar C + F_1 + F_2$, with $F_i \cap \bar C = p_i$ for
$i=1,2$,   $F_1\cap F_2 = \emptyset$, and all three $\bar C,F_1,F_2$ smooth;
\item\label{41-2} $C_0 = \bar C + F$, and 
$\bar C \cap F = \{p_1,p_2\}$,
$\bar C$ and $F$ smooth.
\end{enumerate}

In case~\eqref{41-1}, by generality, the situation must be symmetric with
respect to $p_1$ and $p_2$; hence $F_1$ and $F_2$ are algebraically
equivalent. Let $F$ be their algebraic equivalence class.
Necessarily, $F^2=0$, and $F$ moves in a $1$-dimensional family.
Moreover,
\[
  C\cdot F = \bar C\cdot F = 1; 
\]
hence $F$ is mapped to a line, and $S$ is a scroll, so we have a contradiction.

In case~\eqref{41-2} at least one among $\bar C$ and $F$ dominates the
base $B$; we assume it is $\bar C$, so $\bar C$ has genus at least
one. Since $C$ has arithmetic genus $2$, the only possibility is that
$\bar C$ has genus $1$ and $F$ has genus $0$. Therefore, $F$ is in the
class of the ruling of $\Sigma$. This leads to a contradiction as the
two general points $p_1$ and $p_2$ may not sit on a common ruling.

It thus only remains to examine the finitely many possibilities in
which $S \subset \P^{d-1}$ is 1-weakly defective, listed in
Proposition~\ref{pr:list-1wdef} below.

\eqref{p:l-1w-1}  If $S$ were in the cone over the Veronese surface, then its
  hyperplane section $C$ would be on the Veronese surface itself. But
  then necessarily $S \subset \P^6$, \ie $d=7$, and since all curves
  on the Veronese surface have even degree, this case cannot be
  realized. 

\eqref{p:l-1w-2} In this case $S$ has sectional genus $3$ (the
general hyperplane section of $S$ is birational to a plane quartic
everywhere tangent to the branch curve of the double cover $S
\xrightarrow{2:1} \P^2$), and this is contrary to our assumptions.

\eqref{p:l-1w-3} In this case $S$ is contained in a cone
$K(\Lambda,B) \subset \P^{d-1}$.  Then $B$ is an irreducible curve in
$\P^{d-3}$; hence it has degree $\delta \geq d-3$.  The cone
$K(\Lambda,B)$ is swept out by a $1$-dimensional family $(\Pi_b)_{b
  \in B}$ of planes containing~$\Lambda$.  Let $m$ be the number of
points cut out on $C$ off $\Lambda$ by a general $\Pi_b$.  If $m=1$,
then $S$ is ruled by lines, which is excluded.  Otherwise, let us
consider the section of $C$ by a general hyperplane $H$ containing the
line~$\Lambda$: this is the sum of the points $\Pi_b \cap C$ for $b
\in H\cap B$, plus possibly some points on the vertex line $\Lambda$.
The upshot is that
\[
  d = \deg (H\cap S) \geq m\delta \geq m(d-3),
\]
which is possible only if $d \leq 6$, in contradiction with our
assumptions. Thus this case cannot happen either.

This completes the proof by contradiction that $S$ is rational. Now, let us
write
\[
  |K_{S'}+C| = F + |M|,
\]
where as usual $F$ is the fixed part and $|M|$ the mobile part. Since
$S$ is rational, it cuts out the complete canonical series on $C$; 
hence $|M|$ has dimension $1$, and $F\cdot C=0$ and $M\cdot C=2$. Thus
$|M|$ is mapped to a pencil of conics on $S$, and the proof is complete.
\end{proof}

\begin{proposition}[Chiantini--Ciliberto, \cf \cite{CC-wdef}]
\label{pr:list-1wdef}
Let $S \subset \P^r$, $r \geq 6$, be a $1$-weakly defective surface
with isolated singularities.
Then $S$ falls into one of the following three cases: 
\begin{a-enumerate}
\item\label{p:l-1w-1}  $S \subset \P^6$ is contained in the cone over the Veronese
surface and vertex a point, \ie
\[
  S \subset K\left(p, v_2(\P^2) \right).
\]
\item\label{p:l-1w-2} 
  $S \subset \P^6$ is a quartic double plane $\pi\colon S \xrightarrow
  {2:1} \P^2$ embedded by the complete linear system
  $|\pi^* 2H|$, where $H$ is the line class in $\P^2$.
\item\label{p:l-1w-3}  $S$ is contained in the cone with vertex a line over a curve $B$: 
\[
  S \subset K\left(\Lambda, B \right)
  \subset \P^r.
\]
\end{a-enumerate}
\end{proposition}

\noindent
The proof is a direct application of
\cite[Theorem~1.3]{CC-wdef} and is left to the reader.

\subsection{The Castelnuovo classification}
\label{s:castelnuovo}

As a consequence of Theorem~\ref{t:hyperell-notrat}, we can prove the
following classification result,
from which Corollary~\ref{c:he_class} will follow.

\begin{corollary}
\label{t:castelmoderno}
Let $S$ be as in Theorem~\ref{t:hyperell-notrat}.
If\, $S$ is not a cone, then $S$ is
represented by the complete linear system
\[
  |2E+(g+1+e)F|
\]
on the rational ruled surface $\F_e$, $0\leq e \leq g+1$,
or possibly by a linear subsystem defined by simple base points along
a curvilinear scheme.
\end{corollary}

\begin{proof}
By Theorem~\ref{t:hyperell-notrat}, we know that $S'$ is rational
and contains a pencil $|M|$ of rational curves such that $C\cdot
M=2$.
Then there is a birational morphism $S' \to \Sigma$, with
$\Sigma$ a minimal rational ruled surface on which $|M|$ is mapped to
the ruling. Since $C\cdot M=2$ on $S'$, the linear system $|C|$ is
mapped in $\Sigma$ to a linear system with only simple and double base
points. One may get rid of all double base points by performing
elementary transformations. Thus we may assume that $\Sigma \cong \F_e$
and the general member of $|C|$ (in $\Sigma$) is smooth.

In our usual notation,
$C \lineq 2E + kF$, and one finds $k = g+1+e$ with the adjunction
formula. 
Then,
\begin{equation*}\pushQED{\qed}
  0 \leq C\cdot E = g+1-e.
\qedhere \popQED
	\end{equation*}
\renewcommand{\qed}{} 
\end{proof}

\begin{corollary}
\label{he:castelnuovo}
Let $S$ be as in Theorem~\ref{t:hyperell-notrat}.
If\, $S$ is not a cone, then it is rational and
represented by one of the following linear systems
or one obtained from those by adding simple base points along a
curvilinear subscheme:
\begin{a-enumerate}
\item\label{he:c-1} $\L_{g+3}(g+1,2)$, the linear system of
  plane curves of degree $g+3$ with one base point of multiplicity $g+1$
  and one double point; 
\item\label{case:iperell-infnear}
  $\L_{2g+2-\mu}([2g-\mu, 2^{g-\mu}])$, $\mu=0,\ldots,g$,
  the linear system of
  plane curves of degree $2g+2-\mu$, with one base point $p$ of
  multiplicity $2g-\mu$ and $g-\mu$ double base points infinitely
  near to $p$, 
  pairwise distinct on the exceptional divisor of the blow-up of $p$.
\end{a-enumerate}
\end{corollary}

\begin{proof}
Take into account Corollary~\ref{t:castelmoderno}.
If $e=0$, we are in case~\eqref{he:c-1}.
If $e=1$, we are in case~\eqref{case:iperell-infnear} with $\mu=g$.
If $e>1$, we perform $e-1$ elementary transformations at general
points, thus ending up on $\F_1$, after what we arrive at $\P^2$ by
contracting the $(-1)$-curve $E$ on $\F_1$, with a linear system as in
case~\eqref{case:iperell-infnear} with $\mu = g+1-e$.
\end{proof}

\begin{noname}
\label{p:castelnuovo-storia}
The classification of rational surfaces such that the general
hyperplane section is hyperelliptic
had been classically worked out by Castelnuovo; \cf 
\cite{castelnuovo-iperell}.
In more recent times, the classification of
arbitrary smooth surfaces having
at least one
smooth hyperelliptic hyperplane section has been worked out by
Serrano, \cf \cite{serrano}, 
and Sommese--Van de Ven, \cf \cite{sommese-vdVen}.
Our classification takes into account singular surfaces
as well,
with the same condition that they have one hyperelliptic section,
albeit with the restriction that the degree
$d$ be at least $10$, or at least $2g+3$ if $g=2$ or $3$.
Note that without this assumption, there are other possibilities,
including rational surfaces such that the general hyperplane section
is not hyperelliptic (if the pair $(g,d)$ equals
$(3,8)$ or $(4,9)$) or is an elliptic ruled surface (if
$(g,d)=(3,8)$); \cf \cite{serrano}.
\end{noname}

\section{Gaussian maps and their cokernels}
\label{S:coker}

We consider a linearly normal curve $C \subset \P^r$ and let
$L = \restr {\O_{\P^r}(1)} C$.
In this section we define the Gaussian map $\gamma_{C,L}$ and compute
its corank in a number of cases, thus proving
Theorem~\ref{t:cork-iperell+g3}.

\begin{noname}[The Gaussian map ${\gamma_{C,L}}$]
\label{def:gaussian}
Let $R_{C,L}$ be the kernel of the multiplication map
\begin{equation*}
\label{eq:mult-map}
\mu_{C,L}\colon H^ 0(K_C)\otimes H^ 0(L)\lra 
H^ 0(K_C+L).
\end{equation*}
The Gaussian map
\begin{equation*}
\label{eq:gaussian}
\textstyle
\gamma_{C,L}\colon R_{C,L} \lra H^ 0(2K_C+L)
\end{equation*}
is the map locally defined as
$\sum_i s_i \otimes t_i \mapsto \sum_i (s_i \cdot dt_i- t_i\cdot ds_i)$.
\end{noname}

We are interested in the corank of $\gamma_{C,L}$,
\ie the dimension of its cokernel in $H^ 0(2K_C+L)$.
Most of the time, Castelnuovo's theorem, Theorem~\ref{t:castelnuovo}, 
tells us that the multiplication map $\mu_{C,L}$ is surjective,
which readily gives the dimension of $R_{C,L}$; then it
suffices to compute the dimension of
the kernel of $\gamma_{C,L}$ to find its corank.
To do so, we shall use the canonical identification
which we now explain.

\begin{noname}[Restriction maps to the diagonal]
We shall interpret
the maps $\mu_{C,L}$ and $\gamma_{C,L}$ in terms of operations on the
product $C\times C$. We let $\p_1$ and $\p_2$ be the two 
projections,
\[
  \xymatrix@C=5pt@R=15pt{
    & C \times C
    \ar[dl]_{\p_1} \ar [dr]^{\p_2}
    \\
    C && C\rlap{,}
  }
\]
and $\Delta \subset C\times C$ be the diagonal.
The multiplication map $\mu_{C,L}$ can be identified with the
restriction map $r_1$ to the diagonal $\Delta$, as indicated in the
diagram below: 
\[
  \xymatrix{
    H^0\left(\p_1^*K_C + \p_2^* L\right)
    \ar[d]_{r_1}
    \ar@{=}[r] & H^0(K_C) \otimes H^0(L) \ar[d]^{\mu_{C,L}}
    \\
    H^0\left( \restr {\left(\p_1^*K_C + \p_2^* L\right)} {\Delta} \right)
    \ar@{=}[r] & H^0(K_C+L)
  }
\]
From this we deduce the following identification of the kernel of
the multiplication map $\mu_{C,L}$: 
\[
  R_{C,L}
  \cong
  H^0(C\times C, \p_1^*K_C + \p_2^* L-\Delta).
\]
In turn, the Gaussian map $\gamma_{C,L}$ can be identified with yet again a
restriction 
map to the diagonal, namely the following map $r_2$: 
\[
  r_2\colon
  H^0(C\times C, \p_1^*K_C + \p_2^* L-\Delta)
  \lra 
  H^0 \left(\Delta, \restr {(\p_1^*K_C + \p_2^* L-\Delta)} {\Delta} \right)
  \cong H^0(C, 2K_C+L).
\]
For the last identification, note that
$\restr \Delta \Delta \cong N_{\Delta/C\times C} \cong -K_C$.
The upshot is the following identification of the kernel of
$\gamma_{C,L}$: 
\begin{equation}
\label{eq:ker-gaussian}
  \ker(\gamma_{C,L})
  \cong
  H^0(C\times C, \p_1^*K_C + \p_2^* L-2\Delta).
\end{equation}
\end{noname}

\begin{proposition}
\label{p:he-ker}
Let $C$ be a hyperelliptic curve of genus $g$ and $L$ an
effective line bundle.
There is an identification
\begin{equation}\label{eq:he-ker}
  \ker(\gamma_{C,L} ) \cong H^0\left( C,(g-3) \fg  \right)
  \otimes H^0 \left( C, L-2\fg \right),
\end{equation}
where $\fg$ denotes the class of the $g^1_2$ on $C$.
\end{proposition}

\begin{proof}
We shall prove that the right-hand-sides of 
\eqref{eq:he-ker}
and \eqref{eq:ker-gaussian}
are isomorphic, which suffices to prove the proposition.
First assume that there exists an effective divisor
\[
  E \in \left|\p_1^* K_C + \p_2^*L-2\Delta\right|
\]
in $C \times C$. Then for all $p \in C$, we have the following
equality as divisors on the fibre $\p_2^* p \cong C$: 
\[
  \textstyle
  E\cap \p_2^* p = \sum_{i=1}^{2g-4} p_i,
\]
with the points $p_1,\ldots,p_{2g-4}$ subject to the condition that
\(
2p+\sum_{i=1}^{2g-4} p_i \in |K_C|.
\)
In general, $2p \not\in \fg$, and then there must be two of the points
$p_i$, say $p_1$ and $p_2$, such that $p+p_1,p+p_2 \in \fg$.
Then necessarily $p_1=p_2$, and thus
\[
  \textstyle
  E\cap \p_2^* p = 2p_1+\sum_{i=3}^{2g-4} p_i
  \quad \text{with}
  \quad
  \left\{
    \begin{aligned}
      & p+p_1 \in \fg
      \
      \text{and} \ \\
      & \textstyle
      \sum_{i=3}^{2g-4} p_i \in |K_C-2\fg|=|(g-3)\fg|.
    \end{aligned}
  \right.
\]
This shows that $E$ must contain the graph $I \subset C\times C$ of the
hyperelliptic involution with multiplicity $2$;
in other words, $2I$ is a fixed part of the linear system
$|\p_1^* K_C + \p_2^*L-2\Delta|$, and thus
\[
  H^0\left(C\times C, \p_1^* K_C + \p_2^*L-2\Delta\right)
  = H^0\left(C\times C, \p_1^* K_C + \p_2^*L-2\Delta-2I\right).
\]
The graph $I$ is the
unique divisor in $|\p_1^* \fg+\p_2^* \fg -\Delta|$;  hence
\begin{align}
\label{eq:he-kerr}
  \p_1^* K_C + \p_2^*L-2\Delta
  -2I
  & \lineq
    \p_1^*\left( (g-3)\fg \right)
    + \p_2^* \left( L-2\fg \right),
\end{align}
and the conclusion follows if $\p_1^* K_C + \p_2^*L-2\Delta$
is effective.
On the other hand, the linear equivalence \eqref{eq:he-kerr} implies
that $\p_1^* K_C + \p_2^*L-2\Delta$ is effective if
$\p_1^*( (g-3)\fg )
+ \p_2^* ( L-2\fg )$
is effective;  hence the proposition holds unconditionally. 
\end{proof}

\begin{proposition}
\label{p:g3-ker}
Let $C$ be a non-hyperelliptic curve of genus $3$ and $L$ an
effective line bundle on $C$.
There is an identification
\begin{equation}\label{eq:g3-ker}
  \ker(\gamma_{C,L} ) \cong
  H^0\left( C,L-3K_C\right).
\end{equation}
\end{proposition}

\begin{proof}
As in proof of Proposition~\ref{p:he-ker}, we will identify the
right-hand-side of \eqref{eq:g3-ker} with that of
\eqref{eq:ker-gaussian}, and this will complete the proof.
Assume there exists an effective divisor
\[
  E \in |\p_1^*K_C + \p_2^* L-2\Delta|
\]
on $C\times C$.
Then for all $p\in C$, we have the following
equality as divisors on the fibre $\p_2^* p \cong C$: 
\[
  E \cap \p_2^*(p) = p_1+p_2,
\]
where $p_1,p_2\in C$ are the only two points on $C$ such that
$2p+p_1+p_2 \in K_C$: seeing $C$ as a plane quartic,
$p_1,p_2$ are the residual intersection points of $C$ and its tangent
line at $p$.
Thus $E$ must contain the tangential correspondence
\[
  T = \overline {\left\{
    (q,p) \in C\times C:
    q \neq p \text{ and }
    q \in \T_pC
  \right\}},
\]
where $\T_pC$ denotes the tangent line to $C$ at $p$ in its model as a
plane quartic.
In other words, the linear system
$|\p_1^*K_C + \p_2^* L-2\Delta|$ has $T$ as a fixed part, and therefore
\[
  H^0 \left(
  C\times C,
  \p_1^*K_C + \p_2^* L-2\Delta
  \right)
  =   H^0 \left(
  C\times C,
  \p_1^*K_C + \p_2^* L-2\Delta-T
  \right).
\]

Let us now compute the class of $T$.
For all $(q,p) \in C\times C$, we have the equalities as divisors on
$C$ 
\newcommand{\polar}[2]{\mathrm{D} ^{#1} #2}
\[
  T\cap \p_1^*(q) = C \cap \polar q C - 2q
  \quad \text{and}
  \quad
  T\cap \p_2^*(p) = C \cap \T_pC - 2p,
\]
where $\polar q C$ is the first
polar of $q$ with respect to $C$ (seen as a plane quartic); 
see, \eg
\cite[Appendix~A]{th-formulae}.
First polars with respect to a plane quartic are plane cubics;  hence
$\polar q C$ cuts out the divisor class $3K_C$ on $C$
since the latter is canonically embedded. We thus find
\[
  T\cdot \p_1^*(q) \lineq 3K_C-2q
  \et
  T\cdot \p_2^*(p) \lineq K_C - 2p,
\]
hence
\[
  T \lineq \p_1^* K_C + p_2^* (3K_C) - 2\Delta.
\]
Finally, we find
\begin{equation}
\label{eq:g3-kerr}
  \p_1^*K_C + \p_2^* L-2\Delta-T
  \lineq
  \p_2^*(L-3K_C),
\end{equation}
and this completes the proof if
$\p_1^*K_C + \p_2^* L-2\Delta$ is effective.
On the other hand, it follows from \eqref{eq:g3-kerr}
that $\p_1^*K_C + \p_2^* L-2\Delta$ is effective if
$L-3K_C$ is effective;  hence the result holds unconditionally.
\end{proof}

\begin{proposition}
\label{pr:cork-general}
Let $C$ be a smooth projective curve of genus $g\geq 2$ and $L$ an
effective line bundle on $C$ of degree $d>0$.
\begin{a-enumerate}
\item\label{prop:cork-g1} If\, $C$ is hyperelliptic, then
\[
  \cork(\gamma_{C,L}) =
  2g+2-g\cdot h^1(L) + (g-2)\cdot h^1(L-2\fg)
  - \cork(\mu_{C,L}),
\]
where $\fg$ is the $g^1_2$ on $C$.
\item\label{prop:cork-g3}
 If\, $C$ is non-hyperelliptic of genus $3$, then
\[
  \cork(\gamma_{C,L}) =
  h^0(4K_C-L) -3\cdot h^1(L) 
  - \cork(\mu_{C,L}).
\]
\end{a-enumerate}
\end{proposition}

Note that if $L$ is not effective of positive degree, the multiplication
map $\mu_{C,L}$ has kernel $R_{C,L} = 0$, so that
\[
  \cork(\gamma_{C,L}) = h^0(2K_C+L).
\]
Moreover, we emphasize that the corank of the multiplication map
$\mu_{C,L}$ may be computed in virtually any situation, using for
instance Castelnuovo's theorem, Theorem~\ref{t:castelnuovo}, or the
base-point-free pencil trick; see also \cite[Section~1]{ciliberto83},
\cite[Theorem~(4.e.1)]{green84}, and \cite{pareschi}.
Part~\eqref{prop:cork-g3} of Proposition~\ref{pr:cork-general} had
already appeared as \cite[Proposition~2.9(a)]{kl}.

\begin{proof}
One has
\begin{align*}
  \cork(\gamma_{C,L})
  &= h^0(2K_C+L) -\dim(R_{C,L}) +\dim(\ker \gamma_{C,L}) \\
  &= h^0(2K_C+L) - h^0(K_C)h^0(L) + h^0(K_C+L)
    - \cork(\mu_{C,L}) +\dim(\ker \gamma_{C,L}) \\
  &= (3g-3+d) -g\left(1-g+d + h^1(L)\right)
    +(g-1+d)- \cork(\mu_{C,L}) +\dim(\ker \gamma_{C,L}) \\
  &= (g+4)(g-1)+d(2-g) -g h^1(L) - \cork(\mu_{C,L})
    + \dim(\ker \gamma_{C,L})
\end{align*}
by Riemann--Roch and the fact that $L$ is effective of positive
degree.
For hyperelliptic $C$, one has
\[
  h^0\left( (g-3)\fg \right) = g-2
  \et
  h^0(L-2\fg) = d-g-3+h^1(L-2\fg),
\]
and thus Proposition~\ref{p:he-ker} gives the result.
For $C$ a genus $3$ curve, 
Proposition~\ref{p:g3-ker} gives the result, noting that
\[
  14-d+h^0(C,L-3K_C)=h^0(C,4K_C-L)
\]
by Riemann--Roch and Serre duality. 
\end{proof}

\begin{proof}[Proof of Theorem~\ref{t:cork-iperell+g3}]
Let us first consider the case when $C$ is hyperelliptic.
If $d\geq 2g+3$, then  $L$ is very ample and non-special, 
hence $\mu_{C,L}$ is surjective by Castelnuovo's
theorem, Theorem~\ref{t:castelnuovo};
moreover, $h^1(L-2\fg)=0$ for degree reasons as well, so that
the result follows from Proposition~\ref{pr:cork-general}.

If $d \geq g+4$ and $L$ is general, then
\[
  L \sim p_1+\cdots+p_g + 2\fg + D_0
\]
for some general points $p_1,\ldots,p_g \in C$ and some effective
divisor $D_0$.  In particular, we may assume that $p_1,\ldots,p_g$
impose independent conditions on the canonical series $|K_C|$; hence
$h^1(L)=h^1(L-2\fg)=0$ by Serre duality, and, moreover,
$h^1(L-q)=h^1(L-q-q')=0$ for all $q,q' \in C$.  It follows that $L$ is
very ample; hence $\mu_{C,L}$ is surjective by
Theorem~\ref{t:castelnuovo}, and the result follows from
Proposition~\ref{pr:cork-general}.

We now consider the case when $C$ is non-hyperelliptic of genus $3$.
If $d\geq 2g+1=7$, then $L$ is very ample and non-special;  
hence $\mu_{C,L}$ is surjective by 
Theorem~\ref{t:castelnuovo},
and the result follows from Proposition~\ref{pr:cork-general}.

If $d=2g=6$, then $L$ is base-point-free and non-special, and the map
induced by $|L|$ may identify at most two points $p$ and $q$, which
happens if and only if $L=K_C+p+q$ (all this can be seen by standard
considerations involving the Riemann--Roch theorem). It follows that
$\mu_{C,L}$ is surjective in this case as well, and then
Proposition~\ref{pr:cork-general} gives the result.

If $g+1=4\leq d \leq 5$ and $L$ is general, then $|L|$ is non-special,
and the following may happen:
if $d=5$, $|L|$ is base-point-free and maps $C$ to a plane quintic;
if $d=4$, $|L|$ is a base-point-free $g^1_4$.
In all cases, it follows from Theorem~\ref{t:castelnuovo} that
$\mu_{C,L}$ is surjective, and the result follows from 
Proposition~\ref{pr:cork-general} as in the previous cases.
\end{proof}

\section{Ribbons and extensions}
\label{S:ribb'n'ext}

In this section we interpret the extensions of a smooth polarized curve
$(C,L)$ in terms of the integration of ribbons over $(C,L)$, under the
assumption that it has property $N_2$. This leads to a
necessary condition for $(C,L)$ to be extendable. We also define
universal extensions and give a criterion for their existence.
Most results in this section are essentially an adaptation of some in
\cite{cds}, and we will thus be brief; we also propose various
enhancements with respect to \cite{cds}.

\begin{definition}
\label{d:ext}
Let $(X,L)$ be a polarized variety such that $L$ is very ample, and
consider the projective embedding $X \subset \P^N$ defined by $|L|$.
For all $k\in \N^*$, the polarized variety $(X,L)$ is
\emph{$k$-extendable} if there exist a $Y \subset \P^{N+k}$, not a
cone, and an $N$-dimensional linear subspace
$\Lambda \subset \P^{N+k}$
such that $X = Y\cap \Lambda$.
\end{definition}

In the above situation, we say that $Y$ is a \emph{non-trivial
$k$-extension}, or simply a \emph{non-trivial extension}, of $(X,L)$.
We say that $(X,L)$ is \emph{extendable} if it is $k$-extendable for
some $k>0$.  The trivial extension of $(X,L)$ is defined as the cone
with vertex a point over $X$, in its embedding defined by $|L|$.

\begin{definition}
\label{d:ribb}
Let $(X,L)$ be a polarized variety.
A \emph{ribbon} over $(X,L)$,
also known as a ribbon over $X$ with normal bundle $L$,
is a scheme $\tilde X$ such that $\tilde X_\red = X$ and the ideal
$\I_{X/\tilde X}$ defining $X$ in $\tilde X$ verifies the two
conditions
$\I_{X/\tilde X}^2 = 0$
and
$\I_{X/\tilde X} = L^{-1}$.
\end{definition}

If $X \subset \P^N$ is smooth and $Y \subset \P^{N+1}$ is a
$1$-extension of $(X,\O_X(1))$,
then the first infinitesimal neighbourhood of $X$ in $Y$ is a ribbon
over $(X,\O_X(1))$ which we denote by $2X_Y$.
A ribbon $\tilde X$ over $(X,L)$ is \emph{integrable} if there exists
an extension $Y$ such that $\tilde X = 2X_Y$; in this situation 
we say that the variety $Y$ is an
\emph{integral} of the ribbon $\tilde X$.

A ribbon over $(X,L)$ is uniquely determined by its extension class
$e_{\tilde X} \in \Ext^1(\Omega^1_X,L^{-1})$,
and two ribbons are isomorphic if and only if their extension classes
are proportional.
We will say that a ribbon is \emph{trivial} if its extension class is zero.

Let $(X,L)$ be a smooth polarized manifold with $L$ very ample,
consider the corresponding embedding $X \subset \P^N$, and identify
this $\P^N$ with a hyperplane $H\subset \P^{N+1}$.
If $\tilde X$ is a ribbon over $(X,L)$ contained in the first
infinitesimal neighbourhood $2H_{\P^{N+1}}$, then its extension class
lies in the kernel of the map
\[
  \eta\colon \Ext^1\left(\Omega^1_X,L^{-1}\right) \lra
  \Ext^1\left(\restr {\Omega^1_{\P^N}} X ,L^{-1}\right)
\]
induced by the restriction map
$\restr {\Omega^1_{\P^N}} X \to \Omega^1_X$,
as has been first observed in \cite{voisin-acta}.
When $X$ is a curve, the map $\eta$ can be identified with
$\trsp \gamma _{C,L}$,
the transpose of the Gaussian map defined in Section~\ref{def:gaussian}.

\begin{theorem}
\label{t:unicity}
Let $(C,L)$ be a smooth polarized curve of genus $g\geq 2$
and degree $d\geq 2g+3$.
Then for all $v \in \ker(\trsp \gamma_{C,L})$, the ribbon
$\tilde C_v$ with extension class $v$
is the first infinitesimal neighbourhood of\, $C$ in at most one
surface,
up to automorphism.
In particular, if\, $(C,L)$ is extendable, then $\gamma_{C,L}$ is not
surjective. 
\end{theorem}

The precise meaning of the unicity statement above is the
following. Consider $C \subset \P^{d-g}$ in its embedding defined by
$|L|$, and identify this $\P^{d-g}$ with a hyperplane $H \subset
\P^{d-g+1}$.  Let $S,S' \subset \P^{d-g+1}$ be two surfaces such that
$S\cap H = S'\cap H = C$.  If $2C_S \cong 2C_{S'}$, then there is
exists a projectivity of $\P^{d-g+1}$ acting as the identity on $H$
and mapping $S$ to $S'$.

One gets the necessary condition for the integrability of $(C,L)$ by
applying the unicity statement to the zero vector $0 \in \ker(\trsp
\gamma_{C,L})$. Indeed, the trivial ribbon $\tilde C_0$ is the first
infinitesimal neighbourhood of $C$ in the cone over $C$, so the
unicity statement tells us that if $S$ is a non-trivial extension of
$C$ (thus, $S$ is not a cone), then the two ribbons $2C_S$ and $\tilde
C_0$ are distinct, so $2C_S$ comes from a non-zero vector in
$\ker(\trsp \gamma_{C,L})$.

\begin{noname}
We now outline the proof of Theorem~\ref{t:unicity}, as this will
be needed later on.
It follows a construction given in \cite{wahl97}.
By Theorem~\ref{t:green}, the curve $(C,L)$ has property $N_2$.
Thus the homogeneous ideal of $C$ in its embedding in $\P^{d-g}$
defined by $|L|$ has a minimal resolution as follows: 
\begin{equation}
\label{eq:pres} 
\cdots \longrightarrow
\O_{\P^ {d-g}}(-3)^ {\oplus m_1}
\stackrel {\mathbf r} \longrightarrow   
\O_{\P^ {d-g}}(-2)^ {\oplus m}
\stackrel {\bef}\longrightarrow 
\mathcal I_{C/\P^ {d-g}}\longrightarrow 0.
\end{equation}
We view $\bef$ as a vector of quadratic equations defining $C$
scheme-theoretically, in the homogeneous coordinates $\bx =
(x_0:\ldots: x_{d-g})$ on $\P^{d-g}$.

On the other hand, there is an exact sequence of vector spaces
\begin{equation}
\label{eq:euler}
0 \longrightarrow
H^0(C,L)^\vee  \longrightarrow
H^0(C,N_{C/\P^{d-g}}(-1)) \longrightarrow
\ker \left(\trsp \gamma_{C,L}\right)
\longrightarrow
0; 
\end{equation}
\cf \cite[Lemma~3.2]{cds} and the references given there.
Let $v \in \ker (\trsp \gamma_{C,L})$, and choose a lift of~$v$
in
$H^0(C,N_{C/\P^{d-g}}(-1))$.
By \eqref{eq:pres},  the latter space is a subspace of 
$H^0(C,\O_C(1))^ {\oplus m}$, so we can represent the lift of $v$ as
a length $m$ vector $\bef_v$ of linear forms in the variable $\bx$.

Then the ribbon $\tilde C_v$ with extension class $v$ is the subscheme
of $\P^{d-g+1}$ defined by the equations
\begin{equation}
\label{eq:infdef}
\mathbf f(\mathbf x)+t\mathbf f_v(\mathbf x)=\mathbf 0,\quad  t^
2=0
\end{equation}
in the homogeneous coordinates $(\bx:t)$.
In turn, any surface $S \subset \P^{d-g+1}$ containing $\tilde C_v$ is
defined by the equations
\begin{equation}
\label{eq:integral}
\bef(\bx)+t\bef_v(\bx)+ t^ 2 \bh =\mathbf 0,
\end{equation}
where $\bh$ is a length $m$ vector of constants, subject to conditions
that we will not discuss here (see \cite[Section~4.9]{cds}).
The upshot of these conditions, however, is that two vectors $\bh$ and
$\bh'$ defining two surfaces $S$ and $S'$ containing the ribbon
$\tilde C_v$ differ by an element of
$H^0(C,N_{C/\P^{d-g}}(-2))$. Then the unicity statement follows from
the vanishing of this space when $(C,L)$ has property $N_2$; 
\cf \cite[Lemma~3.6]{cds} and the references given there.
\qed
\end{noname}

One should keep in mind the following conclusions from the above
considerations. The projective space $\P(\ker (\trsp \gamma_{C,L}))$
parametrizes isomorphism classes of non-trivial ribbons over $(C,L)$
likely to be integrated to a non-trivial extension $S$ of $(C,L)$. Each such
ribbon may be integrated to at most one extension, and each $1$-extension
conversely corresponds to a point in $\P(\ker (\trsp \gamma_{C,L}))$.

In analogy with the terminology from deformation theory,\footnote{In
fact, if one looks at the construction in \cite{wahl97}, one sees that
this is more than a mere analogy.}  we will say that the extension
theory of $(C,L)$ is \emph{unobstructed} if every ribbon corresponding
to a point of $\P(\ker (\trsp \gamma_{C,L}))$ is integrable;
otherwise, we say that it is \emph{obstructed}.  When the extension
theory is unobstructed, we will see that we can construct a universal
extension, in the following sense.

\begin{definition}
\label{d:univ}
Let $(C,L)$ be a smooth polarized curve of genus $g\geq 2$
and degree $d\geq 2g+3$.
Let $r = \cork(\gamma_{C,L})$.
An $r$-extension $Y \subset \P^{d-g+r}$ of $(C,L)$ is
\emph{universal} if the following condition holds:
for all $[v] \in \P(\ker(\trsp \gamma_{C,L}))$,
there exists a unique $(d-g+1)$-plane $\Lambda \subset \P^{d-g+r}$
containing $C$ such that the surface $Y\cap \Lambda$ is an integral of
the ribbon over $(C,L)$ defined by the extension class $v$.
\end{definition}

\begin{noname}
\label{p:linear}
Note that, under the above assumptions,
if  $Y \subset \P^{d-g+k}$ is a $k$-extension of $(C,L)$, then it is
defined by equations
\begin{equation*}
  \bef(\bx) + \mathbf{F}(\bx)\cdot \trsp \bt
  + \mathbf{H}( \bt) = 0
\end{equation*}
in homogeneous coordinates $(\bx:\bt)$,
where $\bx=(x_0:\ldots:x_{d-g})$ and
$\bt = (t_1:\ldots:t_k)$,
such that the span $\vect C$ is defined by $\bt=0$;
see for instance \cite[Theorem~20.3]{peeva}.
Here, $\mathbf{F}$ is an $m\times k$ matrix of linear forms in $\bx$,
and 
$\mathbf H$ is a length $m$ vector, constant in $\bx$ and
quadratic in $\bt$.
One thus sees that the map
\begin{equation*}
  \Lambda \in \P^{d-g+k}/ \vect C
  \longmapsto
  [2C_{Y\cap \Lambda}] \in \P(\ker(\trsp \gamma_{C,L}))
\end{equation*}
is linear, given by the matrix $\mathbf{F}$.
(Here $\P^{d-g+k}/ \vect C$ denotes the $(k-1)$-dimensional projective
space of  $(d-g+1)$-planes $\Lambda$ containing $C$,
and the map associates a $(d-g+1)$-plane $\Lambda \subset \P^{d-g+r}$
containing $C$
with the isomorphism class of the ribbon $2C_{Y\cap \Lambda}$).
\end{noname}

\begin{lemma}
\label{l:app_ribb-linear}
Let $(C,L)$ be a smooth polarized curve of genus $g\geq 2$
and degree $d\geq 2g+3$.
Let $Y\subset \P^N$ be an extension of\, $(C,L)$,
of dimension $1+\cork(\gamma_{C,L})$.
Assume that for general
$[e] \in \P(\ker (\trsp \gamma_{C,L}))$, there is a linear subspace
$\Lambda \subset \P^N$ containing $C$ and cutting out a surface on $Y$
such that the ribbon $2C_{Y\cap \Lambda}$ has extension class $e$
$($in other words, $2C_{Y\cap \Lambda} \cong \tilde C_e)$.
Then $Y$ is a universal extension of\, $C$.
\end{lemma}

\begin{proof}
We consider the map
\begin{equation*}
  \Lambda \in \P^{d-g+\cork(\gamma_{C,L})}/ \vect C
  \longmapsto
  [2C_{Y\cap \Lambda}] \in \P\left(\ker\left(\trsp \gamma_{C,L}\right)\right)
\end{equation*}
as in Section~\ref{p:linear} above.
The assumption made on $Y$ means that this map is dominant. It is
moreover linear, as has been observed in Section~\ref{p:linear}. Since
its target and its source are projective spaces of the same dimension, 
namely $\cork(\gamma_{C,L})-1$,
it is an isomorphism, which means that $Y$ is
a universal extension of $(C,L)$.
\end{proof}

\begin{theorem}
\label{t:gnl-ribb}
Let $(C,L)$ be a smooth polarized curve of genus $g\geq 2$
and degree $d\geq 2g+3$.
If the general ribbon in $\P(\ker (\trsp \gamma_{C,L}))$ is
integrable,
then all such ribbons are integrable,
and there exists a universal extension of $(C,L)$.
\end{theorem}

\begin{proof}
Let us first prove, to fix ideas, that if all ribbons in $\P(
\ker (\trsp \gamma_{C,L}))$ are integrable, 
then there exists a universal extension of $(C,L)$.
The proof is identical to that of \cite[Section~5]{cds}, so we will be
brief. The idea is that we can package together all ribbons in 
$\P(\ker (\trsp \gamma_{C,L}))$ and their integrals in a projective
bundle
$\P(\O_{\P(\ker (\trsp \gamma_{C,L}))}^ {\oplus d-g}
\oplus \O_{\P(\ker (\trsp \gamma_{C,L}))}(1))$,
and then the universal extension is the image
of the family of all surface integrals by the map defined by the
relative $\O(1)$ of this projective bundle.
This works as follows.

One first chooses a section 
\[
  v \in \ker \left(\trsp \gamma_{C,L}\right)
  \longmapsto \bef_v \in H^0(C,N_{C/\P^{d-g}}(-1))
\]
of \eqref{eq:euler}.
For all $v \in \ker (\trsp \gamma_{C,L})$, we let $\bh_v$ be the
unique vector of constants such that the integral of the ribbon
$\tilde C_v$ is defined by the equations \eqref{eq:integral} with
$\bh = \bh_v$.
For all $\lambda \in \C$, one has
\begin{equation}
  \label{eq:hom1}
  \bef_{\lambda v}=\lambda \bef_v,
\end{equation}
so the ribbon $\tilde C_{\lambda v}$
(isomorphic to $\tilde C_v$)
is defined by the equations
\begin{equation*}
  \mathbf f(\mathbf x)+t\mathbf f_{\lambda v}(\mathbf x)
  = \bef(\bx)+\lambda t\bef_v(\bx) 
  = \mathbf 0,\quad  t^2
  =0.
\end{equation*}
One may thus deduce the equations of the integral of
$\tilde C_{\lambda v}$
from those of  $\tilde C_v$; namely, they are 
\begin{equation*}
\bef(\bx)+\lambda t\bef_v(\bx)+ \lambda^2 t^ 2 \bh_v =\mathbf 0.
\end{equation*}
By the unicity of the vector of constants $\bh$ attached to $\lambda
v$, we conclude that
\begin{equation}
  \label{eq:hom2}
  \bh_{\lambda v} = \lambda^2 \bh_v.
\end{equation}

Next, let us construct the family
$\cS$ of all surface extensions of $(C,L)$
in $\P(\O_{\P(\ker (\trsp \gamma_{C,L}))}^ {\oplus d-g}
\oplus \O_{\P(\ker (\trsp \gamma_{C,L}))}(1))$
by gluing affine pieces.
We choose a basis $v_1,\ldots,v_r$ of $\ker (\trsp \gamma_{C,L})$, 
$r=\cork(\gamma_{C,L})$,
and for $i=1,\ldots,r$ we consider the subscheme $\cS_i$ of
$\P^{d-g+1} \times \A^{r-1}$ defined by the equations
\[
  \bef (\bx) +
  t\bef(\bx) _{a_1v_1+\cdots+v_i+\cdots+a_rv_r}
  + t^ 2 \bh _{a_1v_1+\cdots+v_i+\cdots+a_rv_r}
  =\mathbf 0  
\]
in the homogeneous coordinates 
$(\bx:t)$ on $\P^{d-g+1}$
and affine coordinates $(a_1,\ldots,\widehat{a_i},\ldots,a_n)$ on
$\A^{r-1}$ (with the convention that the term under the hat should be
omitted);
it is flat over $\A^{r-1}$.
The homogeneity properties
\eqref{eq:hom1} and \eqref{eq:hom2} ensure that any two pieces
$\cS_i$ and $\cS_j$ glue along their open subsets defined by
$(a_j\neq 0)$ and
$(a_i\neq 0)$, via the isomorphism
\[
  ([\bx:t],a_1,\ldots,\widehat{a_i},\ldots,a_n)
  \longmapsto
  \left([\bx:a_i t],\frac {a_1}{a_j},\ldots,\widehat{a_j},\ldots,
  \frac {a_n}{a_j}\right).
\]
The gluing of all $\cS_1,\ldots,\cS_r$ gives the family
$\cS \subset \P(\O_{\P(\ker (\trsp \gamma_{C,L}))}^ {\oplus d-g}
\oplus \O_{\P(\ker (\trsp \gamma_{C,L}))}(1))$
as we wanted.
Finally, the construction of the universal extension as the image
of $\cS$  by the relative $\O(1)$ of the projective bundle is exactly
the same as in \cite[Corollary~5.5]{cds}.

It remains to prove that the integrability of the general ribbon
implies that of all ribbons in $\P(\ker (\trsp \gamma_{C,L}))$.
Let $[v_0] \in \P(\ker (\trsp \gamma_{C,L}))$. If the general ribbon
is integrable, we can find an arc
$\D \subset \P(\ker (\trsp \gamma_{C,L}))$ centred at $[v_0]$ such
that for all $[v] \in \D^\circ = \D -[v_0]$, the corresponding ribbon
$\tilde C_v$ is integrable.
Then we have the following two families, defined by equations as
above:
\begin{itemize}
\item the flat family $\tilde\cC \subset \P^{d-g+1}\times \D$ of all
ribbons $\tilde C_v$, $[v] \in \D$,
and 
\item the flat family $\cS_\D^\circ\subset \P^{d-g+1}\times \D^\circ$
  of the surface integrals of the ribbons
  $\tilde C_v$, $[v]\neq [v_0]$.
\end{itemize}
Taking the closure $\cS_\D$ of $\cS_\D^\circ$ in $\P^{d-g+1}\times \D$, we
obtain a flat family of surfaces over $\D$.
Since $\cS_\D^\circ$ contains $\restr {\tilde\cC} {\D^\circ}$,
$\cS_\D$ will contain ${\tilde\cC}$; 
hence the central fibre of $\cS_\D$ is an integral of the ribbon
$\tilde C_{v_0}$.
Therefore,  all ribbons in $\P(\ker (\trsp \gamma_{C,L}))$ are
integrable, and the theorem is proved.
\end{proof}

\begin{noname}
\label{p:strategy}
In the following sections we shall apply the above
Theorem~\ref{t:gnl-ribb} to various specific situations in which we
know the dimension of $\P(\ker(\trsp \gamma _{C,L}))$.
Our strategy to verify that the general ribbon over $(C,L)$ is
integrable is to produce a family of extensions of $(C,L)$ of the same
dimension as $\P(\ker(\trsp \gamma _{C,L}))$.
Then, by the unicity theorem, Theorem~\ref{t:unicity},
there is an injective map from the parameter space of this family of
extensions to $\P(\ker(\trsp \gamma _{C,L}))$, which 
is dominant for dimension reasons.
\end{noname}

Finally, let us note that all the above considerations may be
adapted to polarized manifolds $(X,L)$ of arbitrary dimension.
The only difference is that the exact sequence
\eqref{eq:euler} should be slightly modified;  see
\cite[Lemma~3.5]{cds}.

\section{Extensions of polarized genus 3 curves}
\label{S:ext-g3}

In this section we study closely the extensions of polarized curves
of genus $3$ and degree $d\geq 2g+3$, in order to determine whether
their ribbons are obstructed or not. Our main output in this direction
is Theorem~\ref{t:ext_g3}.

\begin{noname}[Classification of surfaces with sectional genus 3]
The classification of rational surfaces with hyperplane sections that
are non-hyperelliptic curves of genus $3$ had been classically worked
out by Castelnuovo; \cf \cite{castelnuovo-g3}.  He proved that all
such surfaces are represented by a linear system of plane quartics.
More recently Lanteri and Livorni, \cf \cite{lanteri-livorni}, have
classified all pairs $(S,C)$ where $S$ is a smooth surface, $C \subset
S$ is a smooth genus $3$ curve, and the linear system $|C|$ is
globally generated and ample.  For $d\geq 8$, Theorem~\ref{t:hideg}
provides more generally the classification of surfaces, possibly
singular, with one hyperplane section that is a linearly normal
non-hyperelliptic curve of genus $3$.
\end{noname}

\begin{corollary}
\label{t:ext_g3:rat}
Let $(C,L)$ be a non-hyperelliptic polarized curve of genus $g=3$ and
degree $d\geq 4g-4 =8$. Then 
all surface extensions of\, $(C,L)$ are rational.
If $d\geq 9$, they are all realized by a linear system of plane
quartics; if $d=8$, they are realized either by a linear system of
plane quartics, or by a complete linear system of plane sextics with
seven base points of multiplicity $2$ as in Example~\ref{ex:delp}.
\end{corollary}

\begin{proof}
This is a direct application of Theorem~\ref{t:hideg}.
\end{proof}

\begin{lemma}
\label{l:cork_g3}
Let $C$ be a non-hyperelliptic curve of genus $3$ and $L$ be a
line bundle of degree $d\geq 0$ on it. Then $h^0(C,4K_C-L)$ takes the
following values:
\begin{itemize}
\item If\, $d>16$, then $h^0(4K_C-L)=0$.
\item If\, $d=16$, then $h^0(4K_C-L)=1$ if\, $L=4K_C$ and $0$ otherwise.
\item If\, $d=15$, then $h^0(4K_C-L)=1$ if\, $L=4K_C-p$ for some $p\in C$, and
  $0$ otherwise.
\item If\, $d=14$, then $h^0(4K_C-L)=1$ if\, $L=4K_C-p-q$ for some $p,q\in C$,
  and $0$ otherwise. 
\item If\, $d=13$, then $L$ may always be written as $4K_C-p-q-r$ for some
  $p,q,r\in C$, and $h^0(4K_C-L)=2$ if these three points are aligned on
  the canonical model of $C$, and $1$ otherwise.
\item If\, $d=12$, then $L$ may always be written as $4K_C-p-q-r-s$ for some
  $p,q,r,s\in C$, and $h^0(4K_C-L)=3$ if these four points are aligned on
  the canonical model of $C$, and $2$ otherwise.
\item If\, $d <12$, then $h^0(4K_C-L)=14-d$.
\end{itemize}
\end{lemma}

\begin{proof}
If $d\geq 16$, then $\deg(4K_C-L)\leq 0$ and the result is clear. 
If $d=15$, we can always write $L=4K_C-p_0+N$ for some arbitrarily
chosen point $p_0 \in C$ and some degree $0$ line bundle $N$. If
$h^0(4K_C-L)>0$, then $p_0-N \lineq p$ for some $p\in C$; hence
$L=4K_C-p$. In this case, $h^0(4K_C-L)=h^0(p)=1$.
If $d=14$, it follows as in the previous case that $L=4K_C-p-q$ if
$h^0(4K_C-L)>0$. In this case,
$h^0(4K_C-L)=h^0(p+q)=1$ since $C$ is non-hyperelliptic.
If $d\leq 13$, 
by Jacobi's inversion theorem (see \cite[p.~19]{ACGH}),
we can always write $L = 4K_C-\sum_{i=1}^{16-d} p_i$ for some
points $p_1,\ldots,p_{16-d}$, and then the result follows by
Riemann--Roch and Serre duality.
\end{proof}

\begin{proof}[Proof of Theorem~\ref{t:ext_g3}]
The ``only if'' part of~\eqref{t:ext_g3:cond} is a direct consequence of
Theorem~\ref{t:unicity}, taking into account Theorem
\ref{t:cork-iperell+g3} and Lemma~\ref{l:cork_g3}.
On the other hand, Example~\ref{ex:plane} provides an extension for
all $(C,L)$ with $L = 4K_C - \sum_{i=1}^{16-d} p_i$, thus proving the ``if''
part of~\eqref{t:ext_g3:cond}: explicitly, this goes as follows.
Consider $C$ in its
canonical embedding as a smooth plane quartic,
and let $\epsilon\colon S \to \P^2$ be the blow-up of the plane at the
points $p_1,\ldots,p_{16-d} \in C \subset \P^2$,
with exceptional divisors $E_1,\ldots,E_{16-d}$
(when some $p_i$ coincide, this means that we blow up infinitely
near points).
Let $H$ denote the pull-back to $S$ of the line class on $\P^2$.
The linear system
$|4H - \sum_{i=1} ^{16-d}E_i|$ on $S$
cuts out the complete linear series $|4K_C-\sum_{i=1} ^{16-d}p_i| = |L|$
on the proper transform of~$C$, which is very ample; 
it is thus base-point-free and defines a birational morphism from $S$
to an extension of $(C,L)$.

Let us now prove~\eqref{t:ext_g3:all}.
If $d \geq 14$, there is nothing to add to~\eqref{t:ext_g3:cond} since
in these cases $\cork (\gamma_{C,L})$ is either $0$ or $1$ by
Lemma~\ref{l:cork_g3}, so we will suppose $d\leq 13$.
Given $[Z] \in (\P^2)^{[16-d]}$ for a length $16-d$,
$0$-dimensional subscheme $Z \subset \P^2$,
we let $S_Z$
be the blow-up of $\P^2$ along $Z$, with total
exceptional divisor $E_Z$, and call $H$ the pull-back of the line
class on $\P^2$.
The locus $\cS \subset (\P^2)^{[16-d]}$ parametrizing
those $Z$ such that the
linear system $|4H -E_Z|$ on $S_Z$ contains a smooth
curve is dense.
Moreover, for all $[Z] \in \cS$, since $d\geq 9$,
this linear system has dimension $d-2$, is base-point-free, and
defines a birational morphism.
We consider the universal family
\[
  \mathcal{L} \lra \cS
\]
of these linear systems and the dense open subset
$\L^\circ \subset \L$ consisting of those pairs $(Z,C)$ such that
$C$ is a smooth member of $|4H -E_Z|$ on $S_Z$. After dividing out by
the automorphism group of $\P^2$, we get the moduli space $\cSC$ of
such pairs, which has dimension
\[
  \dim \left( (\P^2)^{[16-d]} \right)
  + (d-2) - 8
  = 22-d.
\]
Next we consider the universal Jacobian $\cJ_3^d$ parametrizing degree
$d$ line bundle on genus $3$ curves,
which has dimension $4g-3 = 9$,
and its dense subset $\cJ^\circ$
corresponding to non-hyperelliptic curves.
We shall examine the map
\[
  c\colon (Z,C) \in \cSC
  \longmapsto
  \left[ C, \restr {\O_{S_Z}(4H-E_Z)} C \right] \in \cJ^\circ,
\]
the fibre of which over a point $(C,L)$ consists of distinct
isomorphism classes of extensions of $(C,L)$.
By our proof of the ``if'' part of~\eqref{t:ext_g3:cond}, the image of $c$
is the locus of those $(C,L)$ such that $L$ may be written as
$4K_C-\sum_{i=1}^{16-d}p_i$,
which is the whole $\cJ^\circ$
since we are assuming $d\leq 13$.
Therefore, all fibres of $c$ have dimension at least
\[
  \dim (\cSC) - \dim (\cJ^\circ)
  = 13-d.
\]
If $d <12$, this proves that for all $(C,L)$, the general ribbon in
$\P(\ker (\trsp \gamma_{C,L}))$ is integrable, by Lemma~\ref{l:cork_g3}
and the argument given in Section~\ref{p:strategy}.
We conclude by Theorem~\ref{t:gnl-ribb} that all ribbons are
integrable and there exists a universal extension. 

If $d=12$ or $13$, the same argument proves~\eqref{t:ext_g3:all}
for all $(C,L)$ such that $L = 4K_C - \sum_{i=1}^{16-d}p_i$ with
$p_1,\ldots,p_{16-d}$ not all on a line in $\P^2$.

We now treat the two remaining cases separately, in the same spirit as
the others.
First assume $d=13$, and let $S_1$ be the blow-up of $\P^2$ along
three pairwise distinct points lying on a line, with total exceptional
divisor $E_1$. 
Consider the locus of smooth quartics in 
the linear system $|4H-E_1|$, which has dimension $11$,
and take its quotient by the automorphism group of $S_1$,
which is the subgroup of the projectivities
of $\P^2$ acting as the
identity on a line, hence has dimension $3$
(it is the group of homotheties and translations of the affine
plane).
We thus get the moduli space $\cSC_1$ of pairs $(S_1,C)$,
of dimension $8$.
The image of $\cSC_1$ by the map $c\colon \L^\circ \to \cJ^\circ$
is the locus $\cJ^1$ of pairs $(C,L)$ such that $L = 4K_C-D$ with $|D|$
a $g^1_3$, which has dimension~$7$
(note that if $p_1,p_2,p_3 \in C \subset \P^2$ are aligned, then they
move in a base-point-free $g^1_3$, and thus up to linear equivalence
we may always assume that they are pairwise distinct).
The upshot is that the fibres of the map
\[
  \restr c {\cSC_1}\colon \cSC_1 \lra \cJ^1
\]
all have dimension at least $1$, and we conclude as in the previous
cases.

In the case $d=12$, we fix three pairwise distinct points
$q_1,q_2,q_3$ on a line in $\P^2$ and let $\cS_2$ be the complement
of $\{ q_1,q_2,q_3 \}$ in this line. For all $q \in \cS_2$, we let
$S_q$ be the blow-up of $\P^2$ at $ q_1,q_2,q_3$, and $q$,
with total exceptional divisor $E_q$. The $S_q$
form a $1$-dimensional family of pairwise non-isomorphic surfaces,
all with automorphism group the
subgroup of the projectivities of $\P^2$ acting as the
identity on a line, which has dimension $3$.
Quotienting the universal family of the linear systems
$|4H-E_q|$, each of dimension $10$, by this automorphism group, we get 
the moduli space $\cSC_2$ of pairs $(S_q,C)$, of dimension
$10+1-3=8$.
The image of $\cSC_2$ by the map $c\colon \L^\circ \to \cJ^\circ$
is the locus $\cJ^2$ of pairs $(C,3K_C)$, which has dimension $6$.
Therefore, the fibres of the map
\[
  \restr c {\cSC_2}\colon \cSC_2 \lra \cJ^2
\]
all have dimension at least $2$, and we conclude as in the previous
cases.
\end{proof}

\begin{remark}
In the case when $L=4K_C-D$ with $|D|$ a $g^1_3$, there exist extensions
of $(C,L)$ supported on a surface different from $S_1$. Indeed,
choosing a member of $|D|$ of the form $2p_1+p_2$, we see that the
blow-up $S_1'$ of $\P^2$ along three aligned points, two of which are
infinitely near, provides an extension of $(C,L)$. This surface is
rigid as $S_1$ is, but its automorphism group is larger: it is the
subgroup of those projectivities fixing
the two points $p_1$ and $p_2$
(hence also leaving the line $\vect{p_1,p_2}$ stable),
which has dimension $4$. We thus
get a moduli space $\cSC_1'$ of dimension $11-4=7$ which surjects onto
$\cJ^1$. Thus for all $(C,L) \in \cJ^1$, there is at least one
extension supported on the surface $S_1'$.

For all $C$, there are finitely many $g^1_3$ having a member of the
form $3p_1$, and the corresponding pairs $(C,L)$ form a
$6$-dimensional locus in $\cJ^\circ$. These pairs have an extension
supported on $S_1''$, the blow-up of $\P^2$ at three infinitely near
points lying on a line. The surface $S_1''$ is rigid and has an
automorphism group of dimension $5$.

Similar considerations may be made about the extensions of the
polarized curves
$(C,3K_C)$. 
\end{remark}

\begin{remark}
One may want to prove Corollary~\ref{t:ext_g3:rat}
for $d\geq 9$
directly with the
above considerations, without resorting to Theorem~\ref{t:hideg}.
Our proof of~\eqref{t:ext_g3:all} shows that for all $(C,L)$, the
general extension of $(C,L)$ is rational and realized by a linear
system of plane quartics. However, it is not clear to us 
why these two properties should be preserved
when one specializes to an
arbitrary extension of $(C,L)$.
\end{remark}

\begin{noname}[Remarks on the extensions in degree ${d<2g+3}$]
\label{p:g3-notN2}
In this case Green's theorem, Theorem~\ref{t:green}, no longer applies to
guarantee that $(C,L)$ has property $N_2$, so that, in particular, 
one ribbon may a priori have several different integrals.

We will mostly concentrate on the case $d=2g+2=8$.
If $L \neq 2K_C$, then all extensions of $(C,L)$ are realized by a
linear system of plane quartics by Corollary~\ref{t:ext_g3:rat},
and the analysis carried out in the proof of Theorem~\ref{t:ext_g3}
applies \textit{mutatis mutandis}. We find that if $(C,L)$ is general, then the
extensions of $(C,L)$ form a family of the expected dimension
\[
  5 = \cork(\gamma_{C,L})-1
  = h^0(4K_C-2L)-1.
\]
We cannot say much more, however, because of the possible failure of
property $N_2$.

The case $L=2K_C$ is more interesting. In this case we still have,
for general $C$, a $5$-dimensional family of extensions of $(C,2K_C)$
given by linear systems of plane quartics: these extensions are
obtained by projecting $v_4(\P^2)$ from eight points that, in $\P^2$,
lie on a conic, so that they sum to a bicanonical
divisor of any quartic containing them.
There is, however, another family of extensions, given by complete linear
systems of plane sextics with seven base points of multiplicity~$2$.
One finds, using similar arguments to those in the proof of
Theorem~\ref{t:ext_g3} above, that for general $C$, these form a
$6$-dimensional family. We thus find two independent families of
extensions of $(C,2K_C)$, one of the expected dimension and one
superabundant.\footnote
{We note in addition that those surfaces obtained by a linear system
  of quartics with eight base points lying on a conic have a
  singularity of type $\frac 1 4(1,1)$; \ie such a surface is locally the cone
  over a rational normal quartic curve.
We also emphasize  that sextics with seven double points are
Cremona-minimal, hence not Cremona-equivalent to smooth quartics; 
see \cite{CalCil-nagoya} and \cite{mella-polastri}.
}

In degree $d<8$ the situation is similar. In fact, linear
systems of sextics with seven double base points of degree $d<8$
necessarily have further base points, so that the funny situation for
$2K_C$ analyzed above now happens for more line bundles (\eg
in degree $7$ it happens for those line bundles that may be written as
$2K_C-p$ for some $p\in C$, which form a $1$-dimensional family in the
Jacobian of $C$).
Note, however, that for $d<8$, we are out of the range of application of
Theorem~\ref{t:hideg} and Corollary~\ref{t:ext_g3:rat} so that there
may be even more families of extensions.
\end{noname}

\begin{noname}[Invariants of the universal extensions]
\label{p:univ-ext-g3}
We now list the degrees and dimensions of the universal extensions of
genus $3$ curves gotten above, limiting ourselves to the case when the
dimension is at least $3$:
\begin{itemize}
\item If $d=13$ and $L= 4K_C-p-q-r$ for some
  $p,q,r\in C$ on a line in the canonical model of $C$,
  we find a threefold $X$ of degree $d=13$ in $\P^{12}$. 
\item If $d=12$ and $L=4K_C-p-q-r-s$ for some
  $p,q,r,s\in C$ not on a line in the canonical model of $C$,
  we find a threefold $X$ of degree $d=12$ in $\P^{11}$.
\item If $d=12$ and $L=3K_C$,
  we find a fourfold $X$ of degree $d=12$ in $\P^{12}$. 
\item If $9 \leq d \leq 11$, we find a $(15-d)$-dimensional variety
  of degree $d$ in $\P^{11}$
  (note that $4\leq 15-d\leq 6$).
\end{itemize}

Smooth projective varieties of degree $d=9,10,11$ have been classified
in 
\cite{fania-livorni9},
\cite{fania-livorni10},
and \cite{besana-biancofiore}, 
respectively.
Running through the corresponding lists, we notice that there is no smooth
$6$-fold of degree $9$ in $\P^{11}$; hence the universal extension in
degree $9$ is certainly singular.
By contrast, in degree $10$ and $11$, there exist a smooth
$5$-fold and a smooth $4$-fold of sectional genus $3$ in $\P^{11}$, namely
$\P^1\times\Q^4$ in the Segre embedding and a scroll over $\P^2$, 
respectively.
It is possible that these coincide with the universal extensions
above, but we will not investigate this now.
\end{noname}

\section{Extensions of polarized hyperelliptic curves}
\label{S:ext-hyperell}

In this section we study closely the extensions of polarized
hyperelliptic curves
of genus $g$ and degree $d\geq 2g+3$ in order to determine whether
their extension theory is obstructed or not. Our main output in this
direction is Theorem~\ref{t:ext_he}.
Our general strategy is similar to that employed for genus $3$ curves
in Section~\ref{S:ext-g3}, but the situation for hyperelliptic curves
is slightly more complicated and requires more care.

\begin{noname}
For all $\mu=g+1,g,\ldots,0$, we let $\H_\mu$ be the complete linear
system
\[
  |2H+\mu F|
  = |2E + (2g+2-\mu)F|
  \quad\text{on }
  \F_{g+1-\mu}.
\]
It defines a projective surface
of degree $4g+4$ in $\P^{3g+5}$, 
which is the maximum possible degree for a
non-trivial extension of a linearly normal curve of genus $g\geq 2$; 
see Corollary~\ref{c:hartshorne}.
Since we focus on polarized curves of degree $d\geq 2g+3$, in order
for property $N_2$ to hold, 
the maximal number of points from which we may project these
surfaces is
\begin{equation}
\label{eq:bmax-he}
  b_{\max} = 2g+1.
\end{equation}
\end{noname}

\begin{proposition}
\label{pr:charact_he}
Let $C$ be a hyperelliptic curve of genus $g$ and
$\fe$ an effective divisor of degree $\mu$ on $C$, $0\leq \mu \leq g+1$.
There exists an embedding of\, $C$ as a member of the linear system
$|2H+\mu F|$ on $\F_{g+1-\mu}$ such that $\restr E C = \fe$
if and only if no two points of $\fe$ are conjugate with respect to
the hyperelliptic involution on $C$.
\end{proposition}

\begin{proof}[Proof of the ``only if'' part of
  Proposition~\ref{pr:charact_he}]
Suppose $C$ is a smooth member of the linear system
$|2H+\mu F|$ on $\F_{g+1-\mu}$, and let
\[
  \restr E C = e_1+\cdots +e_\mu.
\]
Consider two points $e_i$ and $e_j$ of $\restr E C$, distinct in the
sense that $i\neq j$.
Since $E\cdot F=1$ and the $g^1_2$ on $C$ is cut out by $F$, the only
possibility for $e_i$ and $e_j$ to be conjugate is that $e_i=e_j$ and
it is a ramification point of the $g^1_2$. However, the condition that
$e_i=e_j$ means that $C$ is tangent to $E$ at this point, and the
condition that it is a ramification point of the $g^1_2$ means that
$C$ is tangent to the fibre at this point, so these two conditions may
not be realized simultaneously as $C$ is smooth. The conclusion is
thus that for all $i \neq j$, $e_i$ and $e_j$ are not conjugate.
\end{proof}

In case $\mu \leq g$, this can also be proved by cohomological
considerations, as the condition we want to prove is then equivalent
to $h^0(C, \restr E C) = 1$, which in turn is equivalent to
$h^1(\F_{g+1-\mu},E-C)=0$,
as can be seen by considering the restriction exact sequence.
The vanishing holds because $E-C$ has vanishing $h^0$ and $h^2$
(the latter by Serre duality), and $\chi(E-F)=0$ by Riemann--Roch.

\begin{proof}[Proof of the ``if'' part of Proposition~\ref{pr:charact_he}]
Let $C$ be a hyperelliptic curve of genus $g$ and $\fe$ an effective
divisor of degree $\mu$ on $C$, no two points of which are conjugate.
We let $\ff$ be the class of the $g^1_2$ and $\fh = \fe+(g+1-\mu)\ff$.
One has
\[
  K_C-\fh =
  (g-1)\ff-\fe-(g+1-\mu)\ff
  = (\mu-2)\ff -\fe,
\]
so $\fh$ is non-special since our assumption on $\fe$ implies that it
imposes independent conditions to any multiple of $\ff$.
Therefore, $h^0(C,\fh)=g-\mu+3$ by Riemann--Roch.

We first consider the case $\mu \leq g$.
Let $v$ be a basis of $H^0(C,\fe)$
and $s,t$ be a basis of $H^0(C,\ff)$.
Then $vs^{g+1-\mu},\ldots, vt^{g+1-\mu}$ are linearly independent
and span a hyperplane in $H^0(C,\fh)$.
Finally, we choose $u$ such that 
$u,vs^{g+1-\mu},\ldots, vt^{g+1-\mu}$ form a basis of $H^0(C,\fh)$.
Arguing in the same way as above, we find that
$H^0(C,\fh+\ff)$ has dimension $g-\mu+5$
and $us,ut,vs^{g-\mu+2},\ldots, vt^{g-\mu+2}$ is a basis of this
space. 
Thus the linear system $|\fh+\ff|$ maps $C$ in $\P^{g-\mu+4}$ in such
a way that it sits in a rational normal scroll
built on a line and a rational normal curve of degree
$g-\mu+2$ 
spanning complementary subspaces,
that is, $\F_{g+1-\mu}$ in its embedding given by $|H+F|$,
and $\restr F C = \ff$ and $\restr H C = \fh$.
This completes the proof in this case, as $H = E + (g+1-\mu)F$.
The argument is similar when $\mu=g+1$; we leave it to the
reader. 
\end{proof}

\begin{noname}\textit{Proof of Theorem~\ref{t:ext_he}\, \eqref{t:ext_he:ext}.}
\label{p:ext-hell}
Let $(C,L)$ be a polarized hyperelliptic curve of genus $g$ and degree
$d$. 
We first prove the result in the case $d=4g+4$.
We shall see that there are two ways of proceeding in order to write
$L$ as $2\fh+\mu\ff$,
so as to be able to apply Proposition~\ref{pr:charact_he}. We call
these the even and odd ways, respectively 
(note that one way is sufficient to prove~\eqref{t:ext_he:ext}).
We keep the notation $\ff$ for the class of the $g^1_2$ on $C$ and
write $g=2\gamma+\epsilon$ with $\epsilon \in \{0,1\}$ and
$\gamma\in \N$.

In the even way, one first chooses a line bundle
$M^+$ of (even) degree $2g+2$ such that $L=2M^+$.
Then, there is a unique integer $k\in \{0,\ldots,\gamma+\epsilon\}$
such that $M^+$ may be written as
\[
  M^+
  = \fe + (\gamma+1+k)\ff,
\]
with the condition that $\fe$ is the sum of $g+\epsilon-2k$ points
pairwise not conjugate with respect to the $g^1_2$
($k$ is thus the largest integer such that
$M^+ - (\gamma+1+k)\ff$
is effective).
We set
\[
  \mu = g+\epsilon-2k
  \ \text{(even)}
  \quad\text{and}
  \quad
  \fh = \fe + (g+1-\mu)\ff.
\]
By Proposition~\ref{pr:charact_he}, 
there exists an embedding of $C$ as a member of the linear system
$|2H+\mu F|$ on $\F_{g+1-\mu}$ such that $\restr E C = \fe$
and $\restr F C = \ff$.
The normal bundle of $C$ in this embedding is
\begin{align*}
  N_{C/\F_{g+1-\mu}}
  &= 2\fh + \mu \ff \\
  &= 2\fe + (2g+2-\mu) \ff\\
  &= 2\fe+2(\gamma+1+k)\ff = 2M^+ = L; 
\end{align*}
hence the embedding of $\F_{g+1-\mu}$ defined by the complete linear
system $|2H+\mu F|$ is an extension of $(C,L)$, 
as we wanted.

In the odd way, one chooses instead a line bundle $M^-$ of (odd)
degree $2g+1$ such that $L = 2M^-+\ff$.
Then, there is a unique integer  $k\in \{0,\ldots,\gamma\}$
such that $M^-$ may be written as
\[
  M^-
  = \fe + (\gamma+\epsilon+k)\ff,
\]
with $\fe$  the sum of $g+1-\epsilon-2k$ points pairwise not
conjugate with respect to the $g^1_2$.
We set
\[
  \mu = g+1-\epsilon-2k
  \ \text{(odd)}
  \quad\text{and}
  \quad
  \fh = \fe + (g+1-\mu)\ff.
\]
By Proposition~\ref{pr:charact_he}, 
there exists an embedding of $C$ as a member of the linear system
$|2H+\mu F|$ on $\F_{g+1-\mu}$ such that $\restr E C = \fe$
and $\restr F C = \ff$.
The normal bundle of $C$ in this embedding is
\begin{align*}
  N_{C/\F_{g+1-\mu}}
  &= 2\fh + \mu \ff \\
  &= 2\fe + (2g+2-\mu) \ff\\
  &= 2\fe+[2(\gamma+\epsilon+k)+1]\ff = 2M^-+\ff = L,
\end{align*}
which proves that $(C,L)$ is extendable in the same fashion as in the
even way.

It remains to prove~\eqref{t:ext_he:ext} when $d<4g+4$. In this case we
choose $b=4g+4-d$ distinct points $p_1,\ldots,p_b$ on~$C$, to the
effect that
\[
  L^\sharp = L(p_1+\cdots+p_b)
\]
is a line bundle of degree $4g+4$ on $C$. Thus there exists a
non-trivial extension $S^\sharp$ of $(C,L^\sharp)$ by the case
$d=4g+4$, and we get an extension of $(C,L)$ by simple internal
projection of $S^\sharp$ from the points $p_1,\ldots,p_b$, as in
Example~\ref{ex:hyperell}.\qed
\end{noname}

To prove the remaining parts of Theorem~\ref{t:ext_he}, we will apply
Proposition~\ref{pr:charact_he} in the following way.

\begin{corollary}
\label{c:dimlocus_he}
Let $C$ be a hyperelliptic curve of genus $g$,
$\mu \in \{g+1,g,\ldots,0\}$, and $b$ a non-negative integer.
The locus in the Jacobian $J^{4g+4-b}(C)$ of degree $4g+4-b$ line
bundles $L$ on $C$ such that $L$ is the normal bundle of\, $C$ 
in a simple internal projection from $b$ points of the surface
$(\F_{g+1-\mu},2H+\mu F)$ has dimension
$\min(g,\mu+b)$.
\end{corollary}

More precisely, the locus in the corollary parametrizes those $L$ such
that the following holds: 
there exists an embedding of $C$ as a smooth member of the linear
system $|2H+\mu F|$ on $\F_{g+1-\mu}$ passing through $b$ points
$p_1,\ldots,p_b \in \F_{g+1-\mu}$, such that $L$ is the normal bundle
of the proper transform of $C$ in the blow-up of $\F_{g+1-\mu}$
at $p_1,\ldots,p_b$.

\begin{proof}
It follows from Proposition~\ref{pr:charact_he} that the locus
$\mathcal Z \subset J^{2g+2-\mu}(C)$
of degree
$2g+2-\mu$ line bundles $\fh$ such that there exists an embedding of
$C$ as a member of $|2H+\mu F|$ on $\F_{g+1-\mu}$ such that
$\restr H C = \fh$ has dimension $\min(g,\mu)$.
In turn, the locus we are interested in is the image
of $\mathcal Z \times C^{[b]}$, where $C^{[b]}$ denotes the $\supth{b}$ 
symmetric power of $C$, by the map
\begin{equation}
\label{eq:find-mu}
  j_{\mu,b}:
  (\fh,D) \in J^{2g+2-\mu}(C) \times C^{[b]}
  \longmapsto
  2\fh+\mu \ff -D \in J^{4g+4-b}(C),
\end{equation}
with $\ff$ the class of the $g^1_2$; 
hence its dimension is $\min(g,\mu+b)$.
\end{proof}

\begin{proof}[Proof of Theorem~\ref{t:ext_he}\,
\eqref{t:ext_he:univ} and
\eqref{t:ext_he:obstr}]
For all $\mu=g+1,g,\ldots,0$,
 as in the proof of
Theorem~\ref{t:ext_g3}, we consider the universal family
$\cL_{\mu,b} \to \cS_{\mu,b}$ of the linear systems
$|2H+\mu F-E_D|$ on the blow-up of $\F_{g+1-\mu}$ along
$D \in \cS_{\mu,b} = (\F_{g+1-\mu})^{[b]}$, with total exceptional
divisor $E_D$.
The total space $\cL_{\mu,b}$ has dimension
$3g+5+b$.
Then we divide out by the automorphism group of
$\F_{g+1-\mu}$, which has dimension
$5+\max(1,g+1-\mu)$, and end up with a moduli space $\cSC_{\mu,b}$ of
pairs $(S,C)$,  which has dimension
\[
  2g+b-1+\min(g,\mu).
\]
Next we consider the map
\[
  c_{\mu,b}\colon
  (S,C) \in \cSC_{\mu,b}
  \longmapsto
  (C,N_{C/S}) \in \cJ^{4g+4-b}_g
\]
to the universal Jacobian, which by Corollary~\ref{c:dimlocus_he}
surjects onto an irreducible locus of dimension
\[
  2g-1+\min(g,\mu+b)
\]
(recall that the hyperelliptic locus in $\M_g$ has dimension $2g-1$).
Thus, the general fibres of $c_{\mu,b}$ have dimension 
\[
  \delta_{\mu,b}=
  b+\min(g,\mu)-\min(g,\mu+b)
  =
  \begin{cases}
    0 & \text{if } \mu+b\leq g, \\
    b + \mu - g & \text{if } g-b \leq \mu \leq g,  \\
    b  & \text{if } \mu = g+1, \\
  \end{cases}
\]
and all fibres of $c_{\mu,b}$ have dimension at least this.

Now, by Corollary~\ref{t:castelmoderno}, 
the family of all possible extensions of $(C,L)$ is parametrized by
the unions of the fibres of the maps $c_{\mu,b}$ for $\mu=g+1,g,\ldots,0$
and $b=4g+4-d$.
To prove Theorem~\ref{t:ext_he},  it is thus sufficient,
by the argument given in Section~\ref{p:strategy},
to compare the
dimension $\delta_{\mu,b}$ above with the dimension of
$\P(\ker(\trsp \gamma_{C,L}))$, which by Theorem~\ref{t:cork-iperell+g3}
happens to equal $b_{\max}$;  \cf \eqref{eq:bmax-he}.
After that, the conclusion will follow as in the proof of
Theorem~\ref{t:ext_g3}.

One finds that $\delta_{\mu,b} \leq b$ always holds, and thus
\[
  \delta_{\mu,b} = \cork(\gamma_{C,L})-1
  \quad\Iff
  \quad
  b=b_{\max}
  \ \  \text{and}
  \ \ 
  \mu = g+1,g,
\]
which completes the proof of the theorem.
\end{proof}

\begin{noname}[Remarks on the extensions in degree $\boldsymbol{d<2g+3}$]
We can extend the above analysis to degrees lower than $2g+3$,
even though in this range Green's theorem, Theorem~\ref{t:green}, no longer
applies to guarantee that $(C,L)$ has property $N_2$,
and as a consequence many of our results about ribbons and extensions
are no longer usable.

We find that for general $(C,L)$, there is a family of extensions of
dimension
\[
  4g+4-\deg(L)>2g+1,
\]
whereas the corank of $\gamma_{C,L}$ still takes the same value $2g+2$
by Theorem~\ref{t:cork-iperell+g3}, provided we do not consider degrees
lower than $g+4$.
Thus in this case the extensions of $(C,L)$ form a superabundant
family. 
\end{noname}

\begin{noname}[Discussion of universal extensions]
Let us compare the universal extensions 
of degree $2g+3$ and dimension $2g+3$ in $\P^{3g+5}$ gotten in
Theorem~\ref{t:ext_he} with the classification results
in degree $d\leq 11$ available in the literature; \cf
\cite{ionescu,
  fania-livorni9,
  fania-livorni10,
  besana-biancofiore}.
The only relevant cases occur in genus $g=2,3,4$, and no smooth variety
with the appropriate invariants exists in these cases. Therefore, the
corresponding universal extensions are certainly singular.
\end{noname}

\begin{remark}
When $L$ has degree $4g+4$, there are in general only finitely many
extensions of $(C,L)$, even though $\P(\ker(\trsp \gamma_{C,L}))$ has
dimension $2g+1$ (indeed, $\delta_{\mu,0}=0$ in the above notation).
From the arguments given in Section~\ref{p:ext-hell}, it follows that
one can always find at least two different extensions of $(C,L)$, with
underlying surfaces respectively the rational ruled surfaces
$\F_{2k+1-\epsilon}$ (from the even way) and $\F_{2k+\epsilon}$ (from
the odd way); in addition, there are $2^{2g}$ possible choices for
$M^+$ and $M^-$, respectively.  In this case it is clear that there
cannot exist a universal extension of $(C,L)$.

Moreover, it is in general possible to find two line bundles $M$ with the same
parity but leading to different values of $\mu$.
For instance, in the even case, if we have
\[
  M_1 = \fe_1+(\gamma_1+1+k_1)\ff
  \quad\text{and}
  \quad
  M_2 = \fe_2+(\gamma_2+1+k_2)\ff,
\]
the condition that $2M_1=2M_2$ amounts to
\[
  2\fe_1 = 2\fe_2 +(k_2-k_1)\ff,
\]
which is fulfilled if the difference $\fe_1-\fe_2$ consists of
ramification points of the $g^1_2$.

When the degree of $L$ gets smaller than $4g+4$, \ie when the number
$b$ of points from which one projects becomes positive, a given
surface may be obtained by projecting two scrolls with different
values of $\mu$.
Indeed, the sublinear system of $|2H+\mu F|$ on $\F_{g+1-\mu}$ of curves
passing through a point $p$ off $E \subset \F_{g+1-\mu}$
corresponds to the sublinear system of $|2H+(\mu+1) F|$ on $\F_{g-\mu}$
of curves passing through a point $p'$ on $E \subset \F_{g-\mu}$
via the elementary transformation
$\F_{g+1-\mu} \dashrightarrow \F_{g-\mu}$ based at $p$.

Thus one sees that the various loci
$j_{\mu,b}(\mathcal{Z} \times C^{[b]})$
(in the notation of the proof of Corollary~\ref{c:dimlocus_he}) 
are in general not disjoint
as $\mu$ ranges from $0$ to $g+1$.
\end{remark}

\section{Extensions of pluricanonical curves}
\label{S:bican}

In this section we study the extensions of pluricanonical curves,
\ie polarized curves $(C,mK_C)$ with $m>1$.
In particular, we will prove Theorems~\ref{t:cork-plurican}
and~\ref{t:plurican}.
The case of canonical
curves is of a different nature and will not be discussed here; we
refer for instance to \cite{abs1, cds}.

\subsection{Corank of the Gaussian map}
\label{s:cork-plurican}

In this subsection we give the corank of the Gaussian map
$\gamma_{C,mK_C}$ for all non-hyperelliptic curves and $m>1$,
which will prove Theorem~\ref{t:cork-plurican}.
We rely on the following identification.

\begin{noname}
Let $(C,L)$ be a polarized curve, with $L$ very ample, and $N_C$ the
normal bundle in the embedding given by $|L|$.
By \cite[Proposition~1.2]{cm90}, 
we have 
\begin{equation}
\label{eq:cm90}
  \forall m \geq 2:
  \quad
  H^0\left(N_C(-m)\right) \cong
  \ker \left(\trsp \phi_{K_C+(m-1)L, L}\right),
\end{equation}
where for any two line bundles $M,N$ on $C$,
$\phi_{M,N}$ is a Gaussian map
$R(M,N) \to H^0(K_C+M+N)$ defined in a way similar to that in 
Section~\ref{def:gaussian}, with $R(M,N)$ the kernel of the multiplication map
$H^0(M)\otimes H^0(N) \to H^0(M+N)$.

In particular, for $L=K_C$, 
the map $\phi_{K_C+(m-1)K_C, K_C} = \phi_{mK_C, K_C}$ is exactly
the map $\gamma_{C,mK_C}$ defined in Section~\ref{def:gaussian} and
considered in the present paper.
Thus,
as a particular case of \eqref{eq:cm90}, we obtain that 
if $C$ is non-hyperelliptic, then
\begin{equation}
\label{eq:cork-plurican}
  \forall m \geq 2:
  \quad
  H^0\left(N_{C / \P^{g-1}} (-m)\right) \cong
  \ker (\trsp \gamma_{C,mK_C}),
\end{equation}
where the normal bundle
$N_{C / \P^{g-1}}$
is that of the canonical embedding of $C$.
\end{noname}

\begin{proof}[{Proof of Theorem~\ref{t:cork-plurican}}]
The identification \eqref{eq:cork-plurican} will enable us to compute
$\gamma_{C,mK_C}$ in all cases, using the values of
$h^0(N_{C / \P^{g-1}} (-m))$ given in \cite{knutsen18}.
First of all, one has
\begin{equation*}
  h^0(N_{C / \P^{g-1}} (-m))=0
  \quad \text{if}\ 
  m>2
  \ \text{and}\ g\geq 5; 
\end{equation*}
see \cite[Introduction, pp.~58--59]{knutsen18}
and \cite{kl}.
Moreover, one has
\begin{equation*}
  h^0(N_{C / \P^{g-1}} (-m))=0
  \quad \text{if}\ 
  m\geq 2
  \ \text{and}\ \Cliff(C)>2 
\end{equation*}
since canonical curves with Clifford index larger than $2$ have
property $N_2$ by \cite{voisin88,schreyer91}, which is well known to
imply the asserted vanishing;  see
\cite[Lemma~3.6]{cds} and the references given there.

Thus, to prove the theorem,  it only remains to prove that
$\cork(\gamma_{C,2K_C})$ takes the asserted values for curves of
Clifford index $1$ or $2$.
This is readily given in \cite{knutsen18} in the following cases:
\begin{itemize}
\item plane quintics, \cf \cite[Proposition~3.2]{knutsen18}, 
  and sextics, \cf \cite[Proposition~4.3]{knutsen18};
\item tetragonal curves of genus $g\geq 6$ that are not a plane quintic  (this comprises the cases~\eqref{cliff2-b} and~\eqref{cliff2-c} for curves of Clifford index  $2$ in Theorem~\ref{t:cork-plurican}); \cf
  \cite[Proposition~4.1]{knutsen18}.
\end{itemize}
For trigonal curves, \cite[Proposition~3.1]{knutsen18} gives the following
values:
\begin{equation}
\label{eq:3gon-Kn}
  \begin{tabular}{r|l}
    $g$ & $h^0(N_{C/\P^{g-1}} (-2))$ \\
    \hline
    $5$ & 3 \\
    $6$ & 2 \\
    $7$ & $1+h^0(K_C-4\fg) \leq 2$ \\
    $8$ & $h^0(K_C-4\fg) \leq 1$ \\
    $9$ & $h^0(K_C-5\fg) \leq 1$ \\
    $10$ & $h^0(K_C-6\fg) \leq 1$ \\
    $\geq 11$ & 0
  \end{tabular}
\end{equation}
where $\fg$ stands for the class of the $g^1_3$.
We prove in Section~\ref{p:3gon-class}
below that these values always equal
$h^0(K_C-(g-4)\fg)$.
There remains the case of non-trigonal genus $5$ curves, which is
treated in Section~\ref{p:ci-cork} below.
\end{proof}

\begin{noname}[Classification of trigonal curves]
\label{p:3gon-class}
Here we classify trigonal curves of genus $g$ with $5 \leq g \leq 10$
along the lines of \cite[Section~9]{miranda-triple},\footnote{There
is a typo in \cite{miranda-triple}, as the fourth formula at the top of
p.~1153 should read $t=2n-m$ instead of $t=2m-n$.}  in order to prove
that for such curves one has
\[
  \cork(\gamma_{C,2K_C})
  = h^0(N_{C/\P^{g-1}} (-2))
  = h^0(K_C-(g-4)\fg).
\]
This may also be seen with
\cite[Proposition~2.9(e)]{kl},
but the following self-contained proof is more in the spirit of our
text. 
Let $f\colon C \to \P^1$ be a genus $g$ triple cover of $\P^1$.
It holds that
\[
  f_* \O_C = \O_{\P^1} \oplus V, 
\]
with $ V = \O_{\P^1}(-a)\oplus\O_{\P^1}(-b)$ such that
\[
  a+b=g+2
  \quad \text{and}
  \quad
    0<a\leq b
    \leq 2a.
\]
Then $C$ may be realized as a divisor in the rational ruled surface
$\F_{b-a}$, with class
\[
  3E+(2b-a)F
  \lineq 3H+(2a-b)F.
\]
Moreover, one has
\begin{align*}
  h^0\left( (g-4)\fg\right)
  &= h^0 \left(C,f^* \O_{\P^1}(g-4) \right) \\
  &= h^0 \left(\P^1,
    \O_{\P^1}(g-4)
    \oplus \O_{\P^1}(g-4-a)
    \oplus \O_{\P^1}(g-4-b)
    \right) \\
  &= h^0 \left(\P^1,
    \O_{\P^1}(g-4)
    \oplus \O_{\P^1}(b-6)
    \oplus \O_{\P^1}(a-6)
    \right),
\end{align*}
from which one deduces
\begin{align*}
  h^0\left( K_C - (g-4)\fg \right)
  &= h^1\left( (g-4)\fg \right) \\
  &= h^0\left( (g-4)\fg \right) - 3(g-4)+g-1 \\
  &= h^0\left( (g-4)\fg \right) - 2g +11.
\end{align*}
This gives the following complete classification:

\begin{center}
\begin{tabular}{l|cc|c|c|c}
  & $a$ & $b$ & $h^0\left( (g-4)\mathfrak g^1_3 \right)$
  & $h^0\left( K-(g-4)\mathfrak g^1_3 \right)$ & Class \\
  \hline
  $g=5$ & $3$ & $4$ & 2 & 3 & $3E+5F$ on $\F_1$ \\
  \hline
  $g=6$ & 4 & 4 & 3 & 2 & $3E+4F$ on $\F_0$ \\
        & 3 & 5 & 3 & 2 & $3E+7F$ on $\F_2$ \\
  \hline
  $g=7$ & 4 & 5 & 4 & 1 & $3E+6F$ on $\F_1$ \\
        & 3 & 6 & 5 & 2 & $3E+9F$ on $\F_3$ \\
  \hline
  $g=8$ & 5 & 5 & 5 & 0 & $3E+5F$ on $\F_0$ \\
        & 4 & 6 & 6 & 1 & $3E+8F$ on $\F_2$ \\
  \hline
  $g=9$ & 5 & 6 & 7 & 0 & $3E+7F$ on $\F_1$ \\
        & 4 & 7 & 8 & 1 & $3E+10F$ on $\F_3$ \\
  \hline
  $g=10$ & 6 & 6 & 9 & 0 & $3E+6F$ on $\F_0$ \\
  & 5 & 7 & 9 & 0 & $3E+9F$ on $\F_2$ \\
  & 4 & 8 & 10 & 1 & $3E+12F$ on $\F_4$ 
\end{tabular}
\end{center}

From this classification, one readily deduces
that the values given in
\eqref{eq:3gon-Kn}
indeed equal $h^0(K_C-(g-4)\fg)$.
The only non-trivial case is when $g=7$; then one finds as above that
\[
  h^0(4\fg)
  = h^0(\P^1, \O_{P^1}(4) \oplus \O_{P^1}(4-a) \oplus \O_{P^1}(4-b))
  = \begin{cases}
    6 & \text{if}\ (a,b)=(4,5),\\
    7 & \text{if}\ (a,b)=(3,6),
  \end{cases}
\]
hence by Riemann--Roch, 
\[
  h^0(K_C-4\fg)
  = \begin{cases}
    0 & \text{if}\ (a,b)=(4,5),\\
    1 & \text{if}\ (a,b)=(3,6), 
  \end{cases}
\]
as required.
\end{noname}

\begin{noname}[Complete intersection canonical curves]
\label{p:ci-cork}
We conclude the subsection by giving the values of
$\cork(\gamma_{C,mK_C})$ for non-hyperelliptic curves of genus
$g\leq 4$ and non-trigonal curves of genus $5$.

These curves are complete intersections in
their canonical embedding, and thus one finds, using
\eqref{eq:cork-plurican} once again, 
\begin{equation}
\label{eq:cork-g3,4}
\cork(\gamma_{C,mK_C})
=h^0(N_{C/\P^{g-1}}(-m))
= \begin{cases}
    h^0(\O_C(4-m)) & \text{if $g=3$}, \\
    h^0(\O_C(3-m)\oplus\O_C(2-m)) & \text{if $g=4$}, \\
    h^0(\O_C(2-m)^{\oplus 3}) & \text{if $g=5$}. \\
  \end{cases}
\end{equation}
\end{noname}

\subsection{Surface extensions}
\label{s:plurican-ext}

In this subsection we study the surface extensions of pluricanonical
curves. Comparing their number of moduli with the values for
$\cork(\gamma_{C,mK_C})$ found in Section~\ref{s:cork-plurican}, we
will prove Theorem~\ref{t:plurican}, to the effect that the extension
theory of pluricanonical curves having property $N_2$ is
unobstructed. 

We shall prove Theorem~\ref{t:plurican} by considering separately the
various cases for which we have found in
Section~\ref{s:cork-plurican}
above that $\cork(\gamma_{C,mK_C})$ is non-zero
(if $\cork(\gamma_{C,mK_C})$ is zero, then there is only the trivial
ribbon over $(C,mK_C)$, and the statement is empty).
The strategy is the same as for Theorems
\ref{t:ext_g3}
and~\ref{t:ext_he}; 
namely, we prove that the general ribbon in
$\P(\ker(\trsp \gamma_{C,mK_C}))$ is integrable by a dimension count,
after which the conclusion follows by Theorem~\ref{t:gnl-ribb}.
We will be brief and only outline the dimension count.
Note that all extensions of pluricanonical curves appear in the
classification of Theorem~\ref{t:hideg}.

\begin{noname}
Note that
\[
  2g+3 \leq 4g-4
  \quad\Iff
  \quad
  g \geq 4,
\]
so by Green's theorem, Theorem~\ref{t:green}, all pluricanonical curves of
genus $g\geq 4$ have property $N_2$.
Similarly, in genus~$3$, $m$-canonical curves have property $N_2$
if $m \geq 3$.

For bicanonical curves of genus $3$, the sufficient condition for
property $N_2$ provided by Theorem~\ref{t:green} does not hold, and
indeed a direct computation using Macaulay2, \cf \cite{M2}, shows that if
$C \subset \P^5$ is a non-hyperelliptic curve of genus $3$ embedded by
the complete linear series $|2K_C|$, its ideal $\I_{C/\P^5}$ has the
minimal resolution
\[
  0 \lra
  \O_{\P^2}(-6)^{\oplus 3} \lra 
  \O_{\P^2}(-5)^{\oplus 8} \oplus \O_{\P^2}(-4)^{\oplus 3} \lra 
  \O_{\P^2}(-4)^{\oplus 6} \oplus \O_{\P^2}(-3)^{\oplus 8} \lra
                                \O_{\P^2}(-2)^{\oplus 7}
  \lra \I_{C/\P^5} \lra 0,
  \]
so that $(C,2K_C)$ does not have property $N_2$.
\end{noname}

\subsubsection{Complete intersection curves}
 
\begin{noname}[Genus 3]
In this case the result is contained in Theorem~\ref{t:ext_g3}. Let us
briefly recall how it goes in this particular case and add a few
comments.

Let $C$ be a smooth plane quartic.
The $4$-canonical model of $C$ is $v_4(C)$, which is a hyperplane
section of the $4$-Veronese surface $v_4(\P^2)$. Since
$\cork(\gamma_{C,4K_C}) =1$,
we conclude that the unique ribbon over $(C,4K_C)$
is integrable. 

We obtain extensions of the $3$-canonical model of $C$ by projecting
the surface $v_4(\P^2)$ from four points on $v_4(C)$ that, in $\P^2$,
are cut out by a line on $C$, so that they sum to a canonical divisor
of $C$.
There is a $2$-dimensional family of such divisors, and
correspondingly  a $2$-dimensional family of non-isomorphic extensions
of $(C,3K_C)$, in agreement with $\cork(\gamma_{C,3K_C})=3$.
In order to avoid possible misunderstandings, note that the curve
$v_3(C)$ is contained in the $3$-Veronese surface $v_3(\P^2)$ but 
is not a hyperplane section, so $v_3(\P^2)$ is not an extension of
$(C,3K_C)$. 

Bicanonical curves of genus $3$ are out of the range of
Theorem~\ref{t:plurican} because they do not have property
$N_2$. Their extensions have been analyzed in
Section~\ref{p:g3-notN2}, where we have found for general $C$ two
families of extensions, one of the expected dimension $5$ and one
superabundant of dimension $6$.  In this case we cannot apply
Theorem~\ref{t:unicity} to guarantee that each surface is the integral
of a unique ribbon, and thus we have no proof of the fact that the
general ribbon in $\P(\ker(\trsp \gamma_{C,2K_C}))$ is integrable.
\end{noname}

\begin{noname}[Genus 4]
Let $C \subset \P^3$ be a smooth complete intersection of a quadric
$X_2$ and a cubic $X_3$. 
The surface $v_3(X_2) \subset \P^{15}$ is an extension of
$v_3(C)$, in fact the only one since 
$\cork(\gamma_{C,3K_C}) =1$.

Similarly, for each cubic $X'_3$ containing $C \subset \P^3$,
the surface $v_2(X'_3)\subset \P^9$ has $v_2(C)$ as a hyperplane
section, and it is thus an extension of $(C,2K_C)$.
This provides a $4$-dimensional family of extensions of $(C,2K_C)$,
in agreement with $\cork(\gamma_{C,2K_C}) =5$.
\end{noname}

\begin{noname}[Genus 5]
Let $C \subset \P^4$ be a smooth complete intersection of three
quadrics.  For each surface $X_2\cap X_2'$ that is the complete
intersection of two quadrics containing $C$, the surface $v_2(X_2\cap
X_2')$ has $v_2(C)$ as a hyperplane section.  This provides a
$2$-dimensional family of extensions of $(C,2K_C)$, in agreement with
$\cork(\gamma_{C,2K_C}) =3$.
\end{noname}

\subsubsection{Clifford index 1}

\begin{noname}[Trigonal curves]
Let $C$ be a smooth trigonal curve of genus $g\geq 5$,
non-hyperelliptic.
In its canonical model, it sits in a rational normal scroll
$Y \subset \P^{g-1}$ of degree $g-2$,
with class
$C \lineq 3\cH-(g-4)F$, with $\cH$ the hyperplane section class of
$Y \subset \P^{g-1}$ and $F$ the class of a ruling; 
\cf Example~\ref{ex:trig}.
Extensions of $(C,2K_C)$ are to be found as simple projections of the
image of $Y$ by the map $\phi_{|C|}$.
The centre of the projection must be an effective, degree $10-g$,
divisor $D_{10-g}$ on $C$  such that 
\begin{equation*}
  \restr {\left( 3\cH-(g-4)F \right)} C - D_{10-g}
  \lineq 2K_C
  \quad
  \Iff\quad
  D_{10-g} \lineq K_C - (g-4) \fg,
\end{equation*}
where $\fg$ is the class of the $g^1_3$ on $C$
(recall that $\restr \cH C \lineq K_C$).
Thus each $D \in |K_C - (g-4) \fg|$ gives an extension
of $(C,2K_C)$, in agreement with
\(
  \cork(\gamma_{C,2K_C})
  = h^0(K_C-(g-4)\fg).
\)
\end{noname}

\begin{noname}[Plane quintics]
Let $C \subset \P^2$ be a smooth plane quintic.
Then $K_C=\restr {2L} C$, where
$L$ denotes the line class on $\P^2$,
and the extensions of $(C,2K_C)$ are to be found as simple projections
of the Veronese surface $v_5(\P^2)$.
The centre of the projection must be an effective, degree $5$,
divisor $D$ on $C$ such that 
\begin{equation*}
  \restr {5L} C -D \lineq 2K_C
  \quad\Iff
  \quad
  D \lineq \restr L C.
\end{equation*}
Thus each $D \in |\restr L C|$ gives an extension
of $(C,2K_C)$, in agreement with
\(
  \cork(\gamma_{C,2K_C})
  = h^0(\O_{\P^2}(1)).
\)
\end{noname}

\subsubsection{Clifford index 2}

\begin{noname}[Quadric sections of Del Pezzo surfaces]
Let $C$ be a genus $g$ curve such that in its canonical model, $C$ is a
quadric section of a Del Pezzo surface $S \subset \P^{g-1}$. Then
$v_2(S)$ is an extension of $(C,2K_C)$, which proves the theorem in
this case since $\cork(\gamma_{C,2K_C})=1$.
\end{noname}

\begin{noname}[Bielliptic curves]
Let $f\colon C \to E$ be a genus $g$ double cover of the elliptic curve
$E$.
Then the $2$-Veronese re-embedding of the cone in $\P^{g-1}$ over
the elliptic normal curve $E \subset \P^{g-2}$
is an extension of $(C,2K_C)$ as in Example~\ref{ex:biell}, 
which proves the theorem in
this case since $\cork(\gamma_{C,2K_C})=1$.
\end{noname}

\begin{noname}[Plane sextics]
Let $C \subset \P^2$ be a smooth plane sextic. Then $K_C=\restr {3L}
C$, where $L$ denotes the line class on $\P^2$, and the Veronese
surface $v_6(\P^2)$ is an extension of $(C,2K_C)$, 
which proves the theorem in
this case since $\cork(\gamma_{C,2K_C})=1$.
\end{noname}

\subsection{Universal extensions}
\label{s:univ-plurican}

Finally, we consider those pluricanonical curves $(C,mK_C)$ for
which $\cork(\gamma_{C,2K_C}) > 1$ 
and provide a construction of the universal extension, except in the
trigonal case.
The constructions are similar to those in
\cite{cd-higher} and \cite[Appendix]{angelo},
see also \cite{bruno}, and are inspired by examples of Burt Totaro
(private communication, see \cite{cd-double}).
When $\cork(\gamma_{C,2K_C}) = 1$, the universal extension is a
surface, and there is nothing to add to the analysis carried out in  Section~\ref{s:plurican-ext}.

\begin{noname}[Genus 3]
We shall construct a
$4$-dimensional variety $X \subset \P^{12}$ of degree $12$ having
tricanonical curves of genus $3$ as curve sections 
and projections of $v_4(\P^2)$ from four points lying on a line in the
plane as surface sections.

We start from the following basic fact, the proof of which we leave to
the reader.
Let $f,\l \in \C[\bx]$, $\bx=(x_0,x_1,x_2)$, be two homogeneous
polynomials of degrees $4$ and $1$, respectively.
Then the hypersurface $S$ of the weighted projective space $\P(1^3,3)$
defined by the homogeneous equation
\[
  f(\bx)+y\ell(\bx) = 0,
\]
of weighted degree $4$ in the homogeneous coordinates $(\bx,y)$,
is isomorphic to the surface obtained by first blowing up $\P^2$ at
the four points defined by the equations $f(\bx) = \ell(\bx) = 0$ 
and then contracting the proper transform of the line defined by
$\ell(\bx) = 0$.

Now, we claim that the weighted quartic hypersurface
\[
  X:
  \quad
  x_0y_0 + x_1y_1 + x_2y_2 = 0
  \quad\text{in}
  \
  \P(1^3,3^3)_{(\bx:\by)},
\]
in its embedding $X \subset \P^{12}$ defined by weighted cubics,
is the universal extension we are looking for.
To prove our claim, we consider a canonical curve $C$ of genus $3$
defined by a quartic equation $f(\bx)=0$ in $\P^2$. Up to a change of
coordinates, we may assume that $f$ has no term in $x_0^4$, so that it
is possible to write it as
\[
  f = x_1 f_1 + x_2 f_2.
\]
First, we note that $C$ is defined by the three cubic equations
\[
  y_0 = y_1-f_1(\bx) = y_2-f_2(\bx) = 0
\]
in $X$, so that in its tricanonical embedding,  it is the section of $X
\subset \P^{12}$ by three hyperplanes.
Next, we consider a general line $L \subset \P^2$ defined by an
equation 
\[
  x_0 + a_1x_1 +a_2x_2 = 0
\]
in $\P^2$. Then, the extension of $(C,3K_C)$ given by the projection
of $v_4(\P^2)$ from the four points in $C\cap L$ is the section of
$X \subset \P^{12}$ by the two hyperplanes corresponding to the cubic
equations 
\[
  y_1-f_1(\bx) -a_1y_0 = y_2-f_2(\bx) -a_2y_0 = 0 
\]
since the latter is isomorphic to the hypersurface
\[
  (x_1f_1+x_2f_2)+ y_0(x_0+a_1x_1+a_2x_2) = 0
\]
in $\P(1^3,3)$.
It follows that the map
\begin{equation*}
  \Lambda \in \P^{12}/ \vect C
  \longmapsto
  [2C_{X\cap \Lambda}] \in \P(\ker(\trsp \gamma_{C,L})),
\end{equation*}
in the notation of Section~\ref{p:linear}, is dominant.
It thus follows from Lemma~\ref{l:app_ribb-linear} that
$X \subset \P^{12}$ is a universal extension
of $(C,3K_C)$.

One can perform the same construction for bicanonical curves of
genus~$3$, even though they do not have property $N_2$.  Thus, the
weighted quartic hypersurface
\[
  X':
  \quad
  x_0^2y_0 + x_0x_1y_1 + x_0x_2y_2
  + x_1^2y_3 + x_1x_2y_3
  + x_2^2y_5
  =0
  \quad\text{in}
  \
  \P(1^3,2^6)_{(\bx:\by)},
\]
in its embedding in $\P^{11}$ defined by weighted quadrics,
has as surface sections all projections of $v_4(\P^2)$ from eight
points that are the complete intersection of a quartic and a conic.
\end{noname}

\begin{noname}[Genus 4]
We shall construct a $6$-dimensional variety $X \subset \P^{13}$
of degree $12$
having bicanonical curves of genus $4$ and their extensions as
sections by linear spaces.
This is the weighted cubic
\[
  X:
  \quad
  x_0y_0 + x_1y_1 + x_2y_2 + x_3y_3 = 0
  \quad\text{in}
  \
  \P(1^4,2^4)_{(\bx:\by)},
\]
in its embedding $X \subset \P^{13}$ defined by weighted quadrics.

Indeed, let $C$ be the canonical genus $4$ curve defined by the
equations $f,g \in \C[\bx]$ in $\P^3$, of degrees $2$ and~$3$, 
respectively. We may write the degree $3$ equation as
\[
  g = x_0g_0 + x_1 g_1 + x_2 g_2 + x_3g_3,
\]
which enables us to see $C$ as being cut out in $X$ by the five degree
$2$ equations 
\[
  f(\bx)=
  y_0-g_0(\bx)=
  y_1-g_1(\bx)=
  y_2-g_2(\bx)=
  y_3-g_3(\bx)=
  0.
\]
In turn, each cubic surface $X_3' \subset \P^3$ containing $C$ has an
equation of the form $g+(a_0x_0+a_1x_1+a_2x_2+a_3x_3)f$;  hence it is cut
 out in $X$ by the four degree
$2$ equations 
\[
  y_0-g_0(\bx)-a_0f(\bx)=
  y_1-g_1(\bx)-a_1f(\bx)=
  y_2-g_2(\bx)-a_2f(\bx)=
  y_3-g_3(\bx)-a_3f(\bx)=
  0,
\]
so that $v_2(X_3') \subset \P^9$ is the section of
$X \subset \P^{13}$ by four hyperplanes containing $C$, as required.
\end{noname}

\begin{noname}[Genus 5]
The universal extension of non-trigonal bicanonical curves of genus
$5$ is the Veronese $4$-fold $v_2(\P^4) \subset \P^{14}$.
\end{noname}

\begin{noname}[Plane quintics]
The universal extension for bicanonical models of plane quintics may
be constructed analogously to what we have done in the genus~$3$ case
for plane quartics.
Thus the universal extension is the weighted quintic hypersurface
\[
  X:
  \quad
  x_0y_0 + x_1y_1 + x_2y_2 = 0
  \quad\text{in}
  \
  \P(1^3,4^3)_{(\bx:\by)},
\]
in its embedding $X \subset \P^{17}$ defined by weighted cubics, a
projective variety of dimension $4$ and degree $20$.
\end{noname}

\begin{noname}[Bicanonical trigonal curves]
We find a universal extension of dimension greater than $2$ in the
following cases:
\begin{itemize}
\item $g=5$ and $C$ is a member of $|3E+5F|$ on $\F_1$, and
  the universal extension of $(C,2K_C)$ has dimension $4$ and degree
  $16$ in $\P^{14}$.
\item  $g=6$ and $C$ is a member of $|3E+4F|$ on $\F_0$ or of
  $|3E+7F|$ on $\F_2$, and the universal extension of $(C,2K_C)$ has
  dimension $3$ and degree $20$ in $\P^{16}$.
\item  $g=7$ and $C$ is a member of $|3E+9F|$ on $\F_3$, and
  the universal extension of $(C,2K_C)$ has dimension $3$ and degree
  $24$ in $\P^{19}$.
\end{itemize}
We believe it should be possible to give an explicit construction of
these universal extensions along the same lines as in the other cases,
but we do not dwell on this and leave it as an open project.
It is plausible that the universal extension in genus $6$ will be the
same for both kinds of curves.
\end{noname}


\end{document}